\documentclass[a4paper, intlimits, reqno]{amsart}

\usepackage[english]{babel}
\usepackage[T1]{fontenc}
\usepackage[utf8]{inputenc}

\usepackage{amsmath}
\usepackage{amssymb}
\usepackage{MnSymbol}
\usepackage{amsthm}
\allowdisplaybreaks
\usepackage{amsfonts}
\usepackage{mathrsfs} 
\usepackage{mathtools}
\usepackage{nicefrac}
\usepackage{enumitem}

\usepackage{multicol}
\usepackage{url}
\usepackage{dsfont}
\usepackage[numbers,sort&compress]{natbib}
\usepackage{doi}
\usepackage{prettyref}
\usepackage{xcolor}
\usepackage{orcidlink}

\newrefformat{defn}{Definition \ref{#1}}
\newrefformat{rem}{Remark \ref{#1}}
\newrefformat{sect}{Section \ref{#1}}
\newrefformat{sub}{Section \ref{#1}}
\newrefformat{prop}{Proposition \ref{#1}}
\newrefformat{thm}{Theorem \ref{#1}}
\newrefformat{cor}{Corollary \ref{#1}}
\newrefformat{ex}{Example \ref{#1}}
\newrefformat{fig}{Figure \ref{#1}}
\newrefformat{ass}{Assumption \ref{#1}}
\newrefformat{con}{Conjecture \ref{#1}}

\swapnumbers
\newtheoremstyle{dotless}{}{}{\itshape}{}{\bfseries}{}{}{}
\theoremstyle{dotless}
\theoremstyle{plain}
\newtheorem{thm}{Theorem}[section]
\newtheorem{lem}[thm]{Lemma}
\newtheorem{prop}[thm]{Proposition}
\newtheorem{cor}[thm]{Corollary}
\theoremstyle{definition}
\newtheorem{defn}[thm]{Definition}
\newtheorem{rem}[thm]{Remark}
\newtheorem{exa}[thm]{Example}

\newcommand{\N} {\mathbb{N}}

\newcommand{\R} {\mathbb{R}}
\newcommand{\C} {\mathbb{C}}

\newcommand{\ran}{\operatorname{Ran}}

\DeclareMathOperator{\id}{id}
\DeclareMathOperator{\re}{Re}
\providecommand{\differential}{\mathrm{d}}
\renewcommand{\d}{\differential}
\newcommand{\euler}{\mathrm{e}}
\newcommand\rlim{
\mathchoice{\vcenter{\hbox{${\scriptstyle{+}}$}}}
{\vcenter{\hbox{$\scriptstyle{+}$}}}
{\vcenter{\hbox{$\scriptscriptstyle{+}$}}}
{\vcenter{\hbox{$\scriptscriptstyle{+}$}}}}
\newcommand{\vertiii}[1]{{\left\vert\kern-0.25ex\left\vert\kern-0.25ex\left\vert #1 
    \right\vert\kern-0.25ex\right\vert\kern-0.25ex\right\vert}}

\begin{document}

\title[A note on the Lumer--Phillips theorem for bi-continuous semigroups]{A note on the Lumer--Phillips theorem for bi-continuous semigroups}
\author[K.~Kruse]{Karsten Kruse\,\orcidlink{0000-0003-1864-4915}}
\address[Karsten Kruse and Christian Seifert]{Hamburg University of Technology\\ 
Institute of Mathematics \\ 
Am Schwarzenberg-Campus~3 \\
21073 Hamburg \\
Germany}
\email{karsten.kruse@tuhh.de}
\author[C.~Seifert]{Christian Seifert\,\orcidlink{0000-0001-9182-8687}}
\email{christian.seifert@tuhh.de}
\thanks{K.~Kruse acknowledges the support by the Deutsche Forschungsgemeinschaft (DFG) within the Research Training
 Group GRK 2583 ``Modeling, Simulation and Optimization of Fluid Dynamic Applications''.}



\date{\today}
\begin{abstract}
Given a Banach space $X$ and an additional coarser Hausdorff locally convex topology $\tau$ on $X$ we characterise the generators of 
$\tau$-bi-continuous semigroups in the spirit of the Lumer--Phillips theorem, i.e.~by means of dissipativity w.r.t.~a directed system of seminorms and a range condition.
\end{abstract}
\maketitle

\section{Introduction}

Characterising generators of strongly continuous semigroups on Banach spaces is a classical topic due to its relation to well-posedness of the corresponding abstract Cauchy problem \cite[Chap.~II, 6.7 Theorem, p.~150]{engel_nagel2000}. The two main generation theorems go back to Hille--Yosida \cite{hille1948, yosida1948} and Feller--Miyadera--Phillips \cite{feller1952, miyadera1952, phillips1952} for general semigroups and Lumer--Phillips \cite{lumer_phillips1961} for contraction semigroups.

However, there are important examples of semigroups which are not strongly continuous for the Banach space norm, e.g.~the Gau{\ss}--Weierstra{\ss} semigroup on $\mathrm{C}_{\operatorname{b}}(\R^d)$.
To circumvent this issue the concept of bi-continuous semigroups which are strongly continuous only w.r.t.~to a weaker Hausdorff locally convex topology $\tau$ has been introduced by K\"uhnemund \cite{kuehnemund2001}. Thus a natural question is to characterise the generators of bi-continuous semigroups in the spirit of the Hille--Yosida theorem and the Lumer--Phillips theorem. While a version of the Hille--Yosida theorem for bi-continuous semigroups was established directly at the beginning of the theory \cite{kuehnemund2003}, a corresponding version of the Lumer--Phillips theorem for bi-continuous contraction semigroups was missing. Recently, in \cite{budde_wegner2022}, Budde and Wegner introduced the notion of bi-dissipativity to characterise the generators of bi-continuous contraction semigroups. However, the result in \cite{budde_wegner2022} does not cover the key example of the Gau{\ss}--Weierstra{\ss} semigroup on $\mathrm{C}_{\operatorname{b}}(\R^d)$ (see \cite[Example 3.9, p.~8]{budde_wegner2022}), see also \prettyref{rem:bi_dissip_not_reasonable} below. 

In this paper we make use of the notion of $\Gamma$-dissipativity, where $\Gamma$ is a directed set of seminorms generating a topology related to the mixed topology $\gamma\coloneqq \gamma(\|\cdot\|,\tau)$ (see \cite{wiweger1961} and \prettyref{defn:mixed_top_Saks} for the mixed topology), introduced in \cite{albanese2016} to prove versions of the Lumer--Phillips theorem for bi-continuous (contraction) semigroups in \prettyref{thm:gen_bi_cont_contraction} and \prettyref{thm:lumer_phillips_mixed_type}. Note that also \cite{budde_wegner2022} used $\Gamma$-dissipativity to define bi-dissipativity; however, there $\Gamma$ is a fundamental system of seminorms generating the topology $\tau$. Working instead with the mixed topology yields a more natural concept which can also be applied to the Gau{\ss}--Weierstra{\ss} semigroup, see \prettyref{rem:bi_dissip_not_reasonable}.

Let us mention that strongly continuous semigroups have also been considered in locally convex spaces, in particular including a Lumer--Phillips-type generation theorem, see \cite{albanese2016}.

Note that with these results we also answer a question on characterising generators of transition semigroups raised by Markus Kunze, see Problem 3 in \cite[p.~4]{open_problems_OPSO2022}.

In \prettyref{sect:notions} we review the notion of bi-continuous semigroups and their generators as well as the (sub-)mixed topology. We also recall further properties of locally convex spaces as well as operators which we make use of later on. In \prettyref{sect:lumer_phillips} we recall the notion of $\Gamma$-dissipativity and then prove the version of the Lumer--Phillips theorem, where we also comment on its relation to \cite{budde_wegner2022}. Further, we also provide some examples.

\section{Notions and preliminaries}
\label{sect:notions}

In this short section we recall some basic notions and results in the context of bi-continuous semigroups.
For a vector space $X$ over the field $\R$ or $\C$ with a Hausdorff locally convex topology $\tau$ 
we denote by $(X,\tau)'$ the topological linear dual space and just write $X'\coloneqq (X,\tau)'$ 
if $(X,\tau)$ is a Banach space. By $\Gamma_{\tau}$ we always denote a directed system of continuous seminorms that 
generates the Hausdorff locally convex topology $\tau$ on $X$.
Further, for two Hausdorff locally convex spaces $(X,\tau)$ and $(Y,\sigma)$ we use the symbol $\mathcal{L}((X,\tau);(Y,\sigma))$ 
for the space of continuous linear operators from $(X,\tau)$ to $(Y,\sigma)$, and abbreviate $\mathcal{L}(X,\tau)\coloneqq\mathcal{L}((X,\tau);(X,\tau))$.
If $(X,\|\cdot\|_X)$ and $(Y,\|\cdot\|_Y)$ are Banach spaces, we denote by $\tau_{\|\cdot\|_{X}}$ and $\tau_{\|\cdot\|_{Y}}$ the corresponding topologies induced by the norms and just write $\mathcal{L}(X;Y)\coloneqq \mathcal{L}((X,\tau_{\|\cdot\|_X});(Y,\tau_{\|\cdot\|_Y}))$ with operator norm $\|\cdot\|_{\mathcal{L}(X;Y)}$, and $\mathcal{L}(X)\coloneqq \mathcal{L}(X;X)$.

Let us recall the definition of the mixed topology, \cite[Section 2.1]{wiweger1961}, and the notion of a Saks space, 
\cite[I.3.2 Definition, p.~27--28]{cooper1978}, which will be important for the rest of the paper.

\begin{defn}[{\cite[Definition 2.2, p.~3]{kruse_schwenninger2022}, \cite[Proposition 3.11 (a), p.~9]{budde_wegner2022}}]\label{defn:mixed_top_Saks}
Let $(X,\|\cdot\|)$ be a Banach space and $\tau$ a Hausdorff locally convex topology on $X$ that is coarser 
than the $\|\cdot\|$-topology $\tau_{\|\cdot\|}$.  Then
\begin{enumerate}
\item the \emph{mixed topology} $\gamma \coloneqq \gamma(\|\cdot\|,\tau)$ is
the finest linear topology on $X$ that coincides with $\tau$ on $\|\cdot\|$-bounded sets and such that 
$\tau\subseteq \gamma \subseteq\tau_{\|\cdot\|}$; 
\item a directed system of continuous seminorms $\Gamma_{\tau}$ that generates the topology $\tau$ is called \emph{norming} 
if 
\begin{equation}\label{eq:saks}
\|x\|=\sup_{p\in\Gamma_{\tau}} p(x), \quad x\in X;
\end{equation}
\item the triple $(X,\|\cdot\|,\tau)$ is called a \emph{Saks space} if there exists a norming directed system 
of continuous seminorms $\Gamma_{\tau}$ that generates the topology $\tau$.
\end{enumerate}
\end{defn}

The mixed topology is actually Hausdorff locally convex and the definition given above is equivalent to the one 
introduced by Wiweger \cite[Section 2.1]{wiweger1961} due to \cite[Lemmas 2.2.1, 2.2.2, p.~51]{wiweger1961}. 

\begin{defn}[{\cite[Definitions 2.2, 5.4, p.~2, 8]{kruse_seifert2022a}}]
Let $(X,\|\cdot\|,\tau)$ be a Saks space. 
\begin{enumerate}
\item We call $(X,\|\cdot\|,\tau)$ \emph{(sequentially) complete} if $(X,\gamma)$ is (sequentially) complete.
\item We call $(X,\|\cdot\|,\tau)$ \emph{semi-reflexive} if $(X,\gamma)$ is semi-reflexive.
\item We call $(X,\|\cdot\|,\tau)$ \emph{C-sequential} if $(X,\gamma)$ is C-sequential, i.e.~every convex 
sequentially open subset of $(X,\gamma)$ is already open (see \cite[p.~273]{snipes1973}).
\end{enumerate}
\end{defn}

\begin{rem}\label{rem:mixed_norming}
If  $(X,\|\cdot\|,\tau)$ is a sequentially complete Saks space, then $(X,\|\cdot\|,\gamma)$ is also a 
sequentially complete Saks space by \cite[Lemma 5.5, p.~2680--2681]{kruse_meichnser_seifert2018} 
and \cite[Remark 2.3 (c), p.~3]{kruse_schwenninger2022}. In particular, there exists a norming directed system of continuous 
seminorms $\Gamma_{\gamma}$ that generates $\gamma$.
\end{rem}

There is another kind of mixed topology (see \cite[p.~41]{cooper1978}) which becomes quite handy if one has to deal 
with the mixed topology because it is generated by a quite simple directed system of continuous seminorms and often coincides 
with the mixed topology. 

\begin{defn}[{\cite[Definition 3.9, p.~9]{kruse_schwenninger2022}}]\label{defn:submixed_top}
Let $(X,\|\cdot\|,\tau)$ be a Saks space and $\Gamma_{\tau}$ a norming directed system of continuous seminorms 
that generates the topology $\tau$. We set 
\[
\mathcal{N}\coloneqq \{(p_{n},a_{n})_{n\in\N}\;|\;(p_{n})_{n\in\N}\subseteq\Gamma_{\tau},\,(a_{n})_{n\in\N}\in c_{0},
\,a_{n}\geq 0\;\text{for all }n\in\N\}
\]
where $c_0$ is the space of real null-sequences.
For $(p_{n},a_{n})_{n\in\N}\in\mathcal{N}$ we define the seminorm
\[
 \vertiii{x}_{(p_{n},a_{n})_{n\in\N}}\coloneqq\sup_{n\in\N}p_{n}(x)a_{n},\quad x\in X.
\]
We denote by $\gamma_s\coloneqq\gamma_s(\|\cdot\|,\tau)$ the Hausdorff locally convex topology that is generated by 
the system of seminorms $(\vertiii{\cdot}_{(p_n,a_n)_{n\in\N}})_{(p_n,a_n)_{n\in\N}\in\mathcal{N}}$ and call it the \emph{submixed topology}.
\end{defn}

Due to \cite[I.1.10 Proposition, p.~9]{cooper1978}, \cite[I.4.5 Proposition, p.~41--42]{cooper1978} and 
\cite[Lemma A.1.2, p.~72]{farkas2003} we have the following observation.

\begin{rem}[{\cite[Remark 3.10, p.~9]{kruse_schwenninger2022}}]\label{rem:mixed=submixed}
Let $(X,\|\cdot\|,\tau)$ be a Saks space, $\Gamma_{\tau}$ a norming directed system of continuous seminorms 
that generates the topology $\tau$, $\gamma\coloneqq\gamma(\|\cdot\|,\tau)$ the mixed 
and $\gamma_{s}\coloneqq\gamma_{s}(\|\cdot\|,\tau)$ the submixed topology.
\begin{enumerate}
\item[(a)] We have $\tau\subseteq\gamma_s\subseteq \gamma$ and $\gamma_{s}$ has the same convergent 
sequences as $\gamma$.
\item[(b)] If 
 \begin{enumerate}
 \item[(i)] for every $x\in X$, $\varepsilon>0$ and $p\in\Gamma_{\tau}$ there are $y,z\in X$ such that $x=y+z$, 
 $p(z)=0$ and $\|y\|\leq p(x)+\varepsilon$, or 
 \item[(ii)] the $\|\cdot\|$-unit ball $B_{\|\cdot\|}=\{x\in X\;|\; \|x\|\leq 1\}$ is $\tau$-compact,
 \end{enumerate}
then $\gamma=\gamma_s$ holds. 
\end{enumerate}
\end{rem}

The submixed topology $\gamma_s$ was originally introduced in \cite[Theorem 3.1.1, p.~62]{wiweger1961} where a 
proof of \prettyref{rem:mixed=submixed} (b) can be found as well. 

\begin{rem}\label{rem:submixed_norming}
Let $(X,\|\cdot\|,\tau)$ be a Saks space and $\Gamma_{\tau}$ a norming directed system of continuous seminorms 
that generates the topology $\tau$. Then there is a norming directed system of continuous seminorms $\Gamma_{\gamma_s}$ 
that generates the submixed topology $\gamma_{s}=\gamma_{s}(\|\cdot\|,\tau)$. Indeed, we set 
\[
\mathcal{N}_{1}\coloneqq \{(p_{n},a_{n})_{n\in\N}\;|\;(p_{n})_{n\in\N}\subseteq\Gamma_{\tau},\,(a_{n})_{n\in\N}\in c_{0},
\,0\leq a_{n}\leq 1\;\text{for all }n\in\N\}
\]  
and $\Gamma_{\gamma_s}\coloneqq\{\vertiii{\cdot}_{(p_{n},a_{n})_{n\in\N}}\;|\;(p_{n},a_{n})_{n\in\N}\in\mathcal{N}_{1}\}$.
Let $(p_{n})_{n\in\N}\subseteq\Gamma_{\tau}$ and $(a_{n})_{n\in\N}\in c_{0}$ with $a_{n}\geq 0$ for all $n\in\N$. Then 
$C\coloneqq\sup_{n\in\N}a_{n}<\infty$ and w.l.o.g.~$C>0$. We have 
\[
 \vertiii{x}_{(p_{n},a_{n})_{n\in\N}}
=\sup_{n\in\N}p_{n}(x)a_{n}
\leq C\sup_{n\in\N}p_{n}(x)\frac{a_{n}}{C}
=C \vertiii{x}_{\left(p_{n},\tfrac{a_{n}}{C}\right)_{n\in\N}}
\]
for all $x\in X$. In combination with $\mathcal{N}_{1}\subseteq\mathcal{N}$ this shows that 
$\Gamma_{\gamma_s}$ generates $\gamma_s$. Furthermore, for every $p\in\Gamma_{\tau}$ we have with $p_{n}\coloneqq p$ 
for all $n\in\N$ that $p(x)\leq \vertiii{x}_{(p_{n},1/n)_{n\in\N}}$ for all $x\in X$. 
Together with the norming property of $\Gamma_{\tau}$ this implies
\[
 \|x\|
=\sup_{p\in\Gamma_{\tau}}p(x)
\leq \sup_{(p_{n},a_{n})_{n\in\N}\in\mathcal{N}_{1}}\vertiii{x}_{(p_{n},a_{n})_{n\in\N}}
\leq \sup_{(a_{n})_{n\in\N}\in c_{0},\, 0\leq a_{n}\leq 1}\|x\|a_{n}
\leq \|x\|
\]
for all $x\in X$. Hence $\Gamma_{\gamma_s}$ is norming.
\end{rem}

Recall that for a Banach space $(X,\|\cdot\|)$ and a Hausdorff locally convex topology $\tau$ on $X$ the triple $(X,\|\cdot\|,\tau)$ is called \emph{bi-admissible space} if $\tau$ is coarser than $\tau_{\|\cdot\|}$, $\tau$ is sequentially complete on the $\|\cdot\|$-closed unit ball (or equivalently on $\|\cdot\|$-bounded sets), and $(X,\tau)'$ is norming for $X$; cf.\ \cite[Assumption 2.1, p.~3]{budde_wegner2022}.

\begin{lem}
    Let $(X,\|\cdot\|)$ be a Banach space and $\tau$ a Hausdorff locally convex topology on $X$ that is coarser 
than the $\|\cdot\|$-topology $\tau_{\|\cdot\|}$. Then the following are equivalent:
    \begin{enumerate}
     \item[(i)] $(X,\|\cdot\|,\tau)$ is a sequentially complete Saks space.
     \item[(ii)] $(X,\|\cdot\|,\tau)$ is a bi-admissible space.
    \end{enumerate}
\end{lem}

\begin{proof}
    By \cite[Corollary 2.3.2, p.~55]{wiweger1961} a Saks space $(X,\|\cdot\|,\tau)$ is sequentially complete 
    if and only if $(X,\tau)$ is sequentially complete on $\|\cdot\|$-bounded sets, meaning that every $\|\cdot\|$-bounded 
    $\tau$-Cauchy sequence converges in $X$. Combined with \cite[Remark 2.3 (c), p.~3]{kruse_schwenninger2022} it follows
    that a triple $(X,\|\cdot\|,\tau)$ fulfils \cite[Assumptions 1, p.~206]{kuehnemund2003}, which provides a bi-admissible space,
    if and only if it is a sequentially complete Saks space.
\end{proof}

Let us recall the notion of a bi-continuous semigroup.
 
\begin{defn}[{\cite[Definition 3, p.~207]{kuehnemund2003}}]\label{defn:bi_continuous}
Let $(X,\|\cdot\|,\tau)$ be a sequentially complete Saks space. 
A family $(T(t))_{t\geq 0}$ in $\mathcal{L}(X)$ is called $\tau$\emph{-bi-continuous semigroup} if 
\begin{enumerate}
\item[(i)] $(T(t))_{t\geq 0}$ is a \emph{semigroup}, i.e.~$T(t+s)=T(t)T(s)$ and $T(0)=\id$ for all $t,s\geq 0$,
\item[(ii)] $(T(t))_{t\geq 0}$ is $\tau$\emph{-strongly continuous}, 
i.e.~the map $T_{x}\colon[0,\infty)\to(X,\tau)$, $T_{x}(t)\coloneqq T(t)x$, is continuous for all $x\in X$, 
\item[(iii)] $(T(t))_{t\geq 0}$ is \emph{exponentially bounded} (of type $\omega$), i.e.~there exists $M\geq 1$, 
$\omega\in\R$ such that $\|T(t)\|_{\mathcal{L}(X)}\leq M\euler^{\omega t}$ for all $t\geq 0$,
\item[(iv)] $(T(t))_{t\geq 0}$ is \emph{locally bi-equicontinuous}, 
i.e.~for every sequence $(x_n)_{n\in\N}$ in $X$, $x\in X$ 
with $\sup\limits_{n\in\N}\|x_n\|<\infty$ and $\tau\text{-}\lim\limits_{n\to\infty} x_n = x$ it holds that
     \[
      \tau\text{-}\lim_{n\to\infty} T(t)(x_n-x) = 0
     \]
locally uniformly for all $t\in [0,\infty)$.
\end{enumerate}
\end{defn}

\begin{rem}\label{rem:sequentially_mixed_bi_cont}
Let $(X,\|\cdot\|,\tau)$ be a Saks space. 
\begin{enumerate}
\item A sequence in $X$ is $\gamma$-convergent if and only if it is $\|\cdot\|$-bounded and $\tau$-convergent by 
\cite[I.1.10 Proposition, p.~9]{cooper1978}. 
\item Let $(X,\|\cdot\|,\tau)$ be sequentially complete. A semigroup of linear operators $(T(t))_{t\geq 0}$ from $X$ to $X$ 
is $\gamma$-strongly continuous and \emph{locally sequentially $\gamma$-equicontinuous} (i.e.~for all $\gamma$-null sequences $(x_n)_{n\in\N}$ in $X$, $t_0>0$ and $p\in\Gamma_\gamma$ we have $\lim_{n\to\infty}\sup_{t\in[0,t_0]} p(T(t)x_n)=0$)
if and only if it is a $\tau$-bi-continuous semigroup on $X$. 
This follows directly fom part (a) and \cite[Remark 2.6 (b), p.~5]{kruse_schwenninger2022}, and remains true if 
$\gamma$ is replaced by any other Hausdorff locally convex topology on $X$ that has the same convergent sequences as $\gamma$
(cf.~\cite[Proposition A.1.3, p.~73]{farkas2003} for $\gamma$ replaced by $\gamma_s$).
\end{enumerate}
\end{rem} 

We already observed in \prettyref{rem:sequentially_mixed_bi_cont} (b) that 
$\tau$-bi-continuous semigroups are locally sequentially $\gamma$-equicontinuous. 
Under some mild conditions on the Saks space $(X,\|\cdot\|,\tau)$ they are even 
quasi-$\gamma$-equicontinuous. Let us recall what that means.  

\begin{defn}
Let $(X,\upsilon)$ be a Hausdorff locally convex space and $\Gamma_{\upsilon}$ a directed system of continuous seminorms that 
generates $\upsilon$.
A family $(T(t))_{t\in I}$ of linear maps from $X$ to $X$ is called $\upsilon$\emph{-equicontinuous} if  
\[
\forall\;p\in\Gamma_{\upsilon}\;\exists\;\widetilde{p}\in\Gamma_{\upsilon},\;C\geq 0\;
\forall\;t\in I,\,x\in X:\;p(T(t)x)\leq C\widetilde{p}(x).
\]
The family $(T(t))_{t\geq 0}$ is called \emph{locally $\upsilon$-equicontinuous} if $(T(t))_{t\in[0,t_{0}]}$  
is $\upsilon$-equiconti\-nuous for all $t_{0}\geq 0$.
The family $(T(t))_{t\geq 0}$ is called \emph{quasi-$\upsilon$-equicontinuous} if there is $\alpha\in\R$ such that 
$(\euler^{-\alpha t}T(t))_{t\geq 0}$ is $\upsilon$-equicontinuous. Note that one often drops the $\upsilon$ if the topology is clear.
\end{defn}

\begin{rem}\label{rem:relation_gamma_bi_cont_sg}
Let $(X,\|\cdot\|,\tau)$ be a sequentially complete Saks space. Due to \prettyref{rem:sequentially_mixed_bi_cont} (b) 
a $\gamma$-strongly continuous, locally $\gamma$-equicontinuous semigroup of linear operators $(T(t))_{t\geq 0}$ from $X$ to $X$
is a $\tau$-bi-continuous semigroup on $X$. The converse is not true in general by \cite[Example 4.1, p.~320]{farkas2011}. 
However, if $(X,\|\cdot\|,\tau)$ is C-sequential, then the converse also holds 
by \cite[Theorem 3.17 (a), p.~13]{kruse_schwenninger2022}. Even more is true, namely, that every 
$\tau$-bi-continuous semigroup on $X$ is quasi-$\gamma$-equicontinuous if $(X,\|\cdot\|,\tau)$ is C-sequential.
\end{rem}

There is another related notion to equicontinuity on Saks spaces.

\begin{defn}[{\cite[Definitions 3.4, 3.5, p.~6, 7]{kruse_schwenninger2022}}]\label{defn:equitight}
Let $(X,\|\cdot\|,\tau)$ be a Saks space and $\Gamma_{\tau}$ a directed system of continuous
seminorms generating the topology $\tau$. A family of linear maps $(T(t))_{t\in I}$ from $X$ to $X$ is called 
$(\|\cdot\|,\tau)$\emph{-equitight}
if 
\[
\forall\;\varepsilon>0,\,p\in\Gamma_{\tau}\;\exists\;\widetilde{p}\in\Gamma_{\tau},\,C\geq 0\;
\forall\;t\in I,\,x\in X:\;p(T(t)x)\leq C \widetilde{p}(x)+\varepsilon\|x\|.
\]
The family $(T(t))_{t\geq 0}$ is called \emph{locally $(\|\cdot\|,\tau)$-equitight} if 
$(T(t))_{t\in[0,t_{0}]}$ is $(\|\cdot\|,\tau)$-equitight for all $t_{0}\geq 0$.
The family $(T(t))_{t\geq 0}$ is called \emph{quasi-$(\|\cdot\|,\tau)$-equitight} if there is $\alpha\in\R$ such that 
$(\euler^{-\alpha t}T(t))_{t\geq 0}$ is $(\|\cdot\|,\tau)$-equitight.
\end{defn}

At first, tight operators $T\in\mathcal{L}(X)$ as well as families of equitight operators 
$(T(t))_{t\in [0,t_{0}]}$ in $\mathcal{L}(X)$ for $t_{0}\geq 0$ appeared in \cite[Definitions 1.2.20, 1.2.21, p.~12]{farkas2003} 
under the name \emph{local}.
In the setting of $\tau$-bi-continuous semigroups $(T(t))_{t\geq 0}$ the notion of tightness is used in 
\cite[Definition 1.1, p.~668]{es_sarhir2006}, meaning that $(T(t))_{t\in [0,t_{0}]}$ is equitight (or 
local) for all $t_{0}\geq 0$.  
Local equitightness plays an important role in perturbation results for bi-continu\-ous semigroups, 
see e.g.~\cite[Theorem 1.2, p.~669]{es_sarhir2006}, \cite[Theorems 2.4, 3.2, p.~92, 94--95]{farkas2004}, 
\cite[Remark 4.1, p.~101]{farkas2004}, \cite[Theorem 5, p.~8]{budde2021}, \cite[Theorem 3.3, p.~582]{budde2021a}, 
and the corrections regarding \cite{budde2021} in \cite[Remark 3.8, p.~8--9]{kruse_schwenninger2022}.

Due to \cite[Proposition 3.16, p.~12--13]{kruse_schwenninger2022} $(\|\cdot\|,\tau)$-equitightness of a family of linear maps 
$(T(t))_{t\in I}$ from $X$ to $X$ implies $\gamma$-equicontinuity. If $(X,\|\cdot\|,\tau)$ is a sequentially complete 
C-sequential Saks space and $\gamma=\gamma_s$, then any $\tau$-bi-continuous semigroup $(T(t))_{t\geq 0}$ on $X$ is 
quasi-$\gamma$-equicontinuous and quasi-$(\|\cdot\|,\tau)$-equitight 
by \cite[Theorem 3.17, p.~13]{kruse_schwenninger2022}, and both properties are equivalent 
by \cite[Proposition 3.16, p.~12--13]{kruse_schwenninger2022}.

We close this section by recalling the definition of the generator of a $\tau$-bi-continuous semigroup and 
two of its properties which we will need.

\begin{defn}[{\cite[Definition 1.2.6, p.~7]{farkas2003}}]\label{defn:bi_cont_generator}
Let $(X,\|\cdot\|,\tau)$ be a sequentially complete Saks space
and $(T(t))_{t\geq 0}$ a $\tau$-bi-continuous semigroup on $X$. The \emph{generator} $(A,D(A))$ 
is defined by
\begin{align*}
D(A)\coloneqq &\Bigl\{x\in X\;|\;\tau\text{-}\lim_{t\to 0\rlim}\frac{T(t)x-x}{t}\;\text{exists in } X\;
\text{and }\sup_{t\in(0,1]}\frac{\|T(t)x-x\|}{t}<\infty\Bigr\},\\
Ax\coloneqq &\tau\text{-}\lim_{t\to 0\rlim}\frac{T(t)x-x}{t},\quad x\in D(A).
\end{align*}
\end{defn}

\begin{prop}[{\cite[Corollary 13, p.~215]{kuehnemund2003}}]\label{prop:generator}
Let $(X,\|\cdot\|,\tau)$ be a sequentially complete Saks space   
and $(T(t))_{t\geq 0}$ a $\tau$-bi-continuous semigroup on $X$ with generator $(A,D(A))$. 
Then the following assertions hold:
\begin{enumerate}
\item The generator $(A,D(A))$ is \emph{bi-closed}, i.e.~whenever $(x_{n})_{n\in\N}$ is a sequence in 
$D(A)$ such that $\tau\text{-}\lim_{n\to\infty}x_{n}=x$ and $\tau\text{-}\lim_{n\to\infty}Ax_{n}=y$ 
for some $x,y\in X$ and both sequences are $\|\cdot\|$-bounded, then $x\in D(A)$ and $Ax=y$.
\item The generator $(A,D(A))$ is \emph{bi-densely defined}, i.e.~for each $x\in X$ there exists a $\|\cdot\|$-bounded sequence 
$(x_{n})_{n\in\N}$ in $D(A)$ such that $\tau\text{-}\lim_{n\to\infty}x_{n}=x$.
\end{enumerate}
\end{prop}

\section{Lumer--Phillips for bi-continuous semigroups} 
\label{sect:lumer_phillips}

First, we recall the relevant notions from \cite{albanese2016} concerning dissipative linear operators 
on Hausdorff locally convex spaces. 
We write in short that $(A,D(A))$ is a linear operator on a Hausdorff locally convex space $X$ if $A\colon D(A)\subseteq X\to X$ is 
a linear operator. 

\begin{defn}[{\cite[Definitions 3.1, 3.5, p.~923]{albanese2016}}]
Let $(X,\upsilon)$ be a Hausdorff locally convex space and $(A,D(A))$ a linear operator on $X$. 
\begin{enumerate}
\item $(A,D(A))$ is called \emph{$\upsilon$-closed} if for each net $(x_{i})_{i\in I}\subseteq D(A)$ 
satisfying $x_{i}\to x$ and $Ax_{i}\to y$ w.r.t.~$\upsilon$ for some $x,y\in X$, we have $x\in D(A)$ and $Ax=y$.
\item A linear operator $(B,D(B))$ on $X$ is called an \emph{extension} of $(A,D(A))$ if $D(A)\subseteq D(B)$ and $B_{\mid D(A)}=A$.  
The operator $(A,D(A))$ is called \emph{$\upsilon$-closable} if it admits a $\upsilon$-closed extension. 
The smallest $\upsilon$-closed extension of an $\upsilon$-closable operator $(A,D(A))$ is called the \emph{$\upsilon$-closure} 
of $(A,D(A))$ and denoted by $(\overline{A},D(\overline{A}))$.
\item $(A,D(A))$ is called \emph{(sequentially) $\upsilon$-densely defined} if $D(A)$ is (sequentially) $\upsilon$-dense in $X$.
\item Let $(A,D(A))$ be $\upsilon$-densely defined. The \emph{$\upsilon$-dual operator} $(A',D(A'))$ of $(A,D(A))$ on $(X,\upsilon)'$ 
is defined by setting
\[
D(A')\coloneqq\{x'\in(X,\upsilon)'\;|\;\exists\;y'\in(X,\upsilon)'\;\forall\;x\in D(A):\;\langle Ax,x' \rangle=\langle x, y' \rangle\}
\]
and $A'x'\coloneqq y'$ for $x'\in D(A')$.
\end{enumerate}
\end{defn}

We have the following relation to the notions of a bi-closed resp.~bi-densely defined operator.

\begin{rem}\label{rem:mixed_closed_densely_defined}
Let $(X,\|\cdot\|,\tau)$ be a Saks space and $(A,D(A))$ a linear operator. 
Due to \prettyref{rem:sequentially_mixed_bi_cont} (a) $(A,D(A))$ is bi-closed (see \prettyref{prop:generator} (a)) 
if and only if it is sequentially $\gamma$-closed (which is defined analogously to $\gamma$-closedness but with nets 
replaced by sequences). 
Again by \prettyref{rem:sequentially_mixed_bi_cont} (a) $(A,D(A))$ is bi-densely defined 
(see \prettyref{prop:generator} (b)) if and only if it is sequentially $\gamma$-densely defined. 
Moreover, if $(A,D(A))$ is sequentially $\gamma$-densely defined, then it is obviously $\gamma$-densely defined, too (as every sequence is a net).
\end{rem}

\begin{defn}[{\cite[p.~922]{albanese2016}}]
Let $(X,\upsilon)$ be a Hausdorff locally convex space and $(A,D(A))$ a linear operator on $X$. 
If $\lambda\in\C$ is such that $\lambda-A\coloneqq \lambda\id -A\colon D(A)\to X$ is injective, then the linear operator 
$(\lambda-A)^{-1}$ exists and is defined on the domain $\ran(\lambda-A)\coloneqq \{(\lambda-A)x\;|\;x\in D(A)\}$, 
i.e.~the range of $\lambda-A$. The \emph{resolvent set} of $A$ is defined by 
\[
\rho_{\upsilon}(A)\coloneqq\{\lambda\in\C\;|\;\lambda-A\text{ is bijective and }(\lambda-A)^{-1}\in\mathcal{L}(X,\upsilon)\}.
\]
If $\upsilon=\tau_{\|\cdot\|}$ for a Banach space $(X,\|\cdot\|)$, we just write $R(\lambda,A)\coloneqq (\lambda-A)^{-1}$ 
and $\rho(A)\coloneqq\rho_{\upsilon}(A)$. 
\end{defn}

\begin{defn}[{\cite[Definition 3.9, p.~925]{albanese2016}}]
Let $(X,\upsilon)$ be a Hausdorff locally convex space and $\Gamma_{\upsilon}$ a directed system of continuous seminorms that 
generates $\upsilon$. A linear operator $(A,D(A))$ on $X$ is called 
$\Gamma_{\upsilon}$\emph{-dissipative} if 
\[
\forall\;\lambda>0,\,x\in D(A),\,p\in\Gamma_{\upsilon}:\;p((\lambda-A)x)\geq\lambda p(x).
\]
\end{defn}

It is important to note that in contrast to equicontinuity or equitightness the notion of dissipativity depends 
on the selection of the directed system of continuous seminorms 
that generates the topology $\upsilon$ by \cite[Remark 3.10, p.~925--926]{albanese2016}.

\begin{rem}\label{rem:norming_diss_Banach_diss}
Let $(X,\|\cdot\|,\tau)$ be a Saks space, $\upsilon$ a Hausdorff locally convex topology on $X$ 
and $(A,D(A))$ a $\Gamma_{\upsilon}$-dissipative operator on $X$. If $\Gamma_{\upsilon}$ is norming, then it 
follows from the $\Gamma_{\upsilon}$-dissipativity and \eqref{eq:saks} that
\[
\forall\;\lambda>0,\,x\in D(A):\;\|(\lambda-A)x\|\geq\lambda \|x\|.
\]
Thus $(A,D(A))$ is also a dissipative operator on the Banach space $(X,\|\cdot\|)$ in the sense of 
\cite[Chap.~II, 3.13 Definition, p.~82]{engel_nagel2000} (cf.~\cite[Remark 3.3 (i), p.~5]{budde_wegner2022} for $\upsilon=\tau$). 
We also denote such kind of dissipativity on a Banach space $(X,\|\cdot\|)$ by $\|\cdot\|$\emph{-dissipativity}.
\end{rem}

In \cite{budde_wegner2022} another notion of dissipativity on Saks spaces was introduced. 

\begin{rem}\label{rem:bi_dissip_submixed}
Let $(X,\|\cdot\|,\tau)$ be a sequentially complete Saks space. In \cite[Definition 3.2, p.~5]{budde_wegner2022} 
a linear operator $(A,D(A))$ on $X$ is called \emph{bi-dissipative} if there exists a norming directed system of 
continuous seminorms $\Gamma_{\tau}$ that generates $\tau$ such that $(A,D(A))$ is $\Gamma_{\tau}$-dissipative. 
It is then shown in the proof of \cite[Theorem 3.15, p.~11]{budde_wegner2022} that a bi-dissipative operator $(A,D(A))$ 
is also $(\vertiii{\cdot}_{(p_n,a_n)_{n\in\N}})_{(p_n,a_n)_{n\in\N}\in\mathcal{N}}$-dissipative 
since 
\[
 \vertiii{(\lambda-A)x}_{(p_{n},a_{n})_{n\in\N}}
=\sup_{n\in\N}p_{n}((\lambda-A)x)a_{n}
\geq \sup_{n\in\N}\lambda p_{n}(x)a_{n}
=\lambda \vertiii{x}_{(p_{n},a_{n})_{n\in\N}}
\]
for all $\lambda>0$, $x\in D(A)$ and $(p_n,a_n)_{n\in\N}\in\mathcal{N}$, where 
$(\vertiii{\cdot}_{(p_n,a_n)_{n\in\N}})_{(p_n,a_n)_{n\in\N}\in\mathcal{N}}$ is the system of seminorms that generates 
the submixed topology $\gamma_s$ from \prettyref{defn:submixed_top}. We observe that this also implies that 
$(A,D(A))$ is $\Gamma_{\gamma_s}$-dissipative w.r.t~the norming directed system of continuous seminorms $\Gamma_{\gamma_s}$
that generates the submixed topology $\gamma_s$ from \prettyref{rem:submixed_norming}.
\end{rem}

\begin{prop}\label{prop:lambda_A}
Let $(X,\|\cdot\|,\tau)$ be a sequentially complete Saks space, $\upsilon$ a Hausdorff locally convex topology on $X$, 
$(A,D(A))$ a $\Gamma_{\upsilon}$-dissipative operator on $X$. Then the following assertions hold:
\begin{enumerate}
\item $\lambda-A$ is injective for all $\lambda>0$. Moreover, we have 
\begin{equation}\label{eq:res_cont}
\forall\;\lambda>0,\,x\in\ran(\lambda-A),\,p\in\Gamma_{\upsilon}:\;p((\lambda-A)^{-1}x)\leq \frac{1}{\lambda} p(x).
\end{equation}
\item If $\ran(\lambda-A)$ is (sequentially) $\upsilon$-closed for some $\lambda>0$, then $(A,D(A))$ is (sequentially) 
$\upsilon$-closed. If $\upsilon=\gamma$, then the converse even holds for all $\lambda>0$.
\item Let $\Gamma_{\upsilon}$ be norming. Then $\lambda-A$ is surjective for some $\lambda>0$ if and only 
if it is surjective for all $\lambda>0$. In such a case, $(0,\infty)\subseteq \rho(A)$.
\item Let $\upsilon=\gamma$. Then $\lambda-A$ is surjective for some $\lambda>0$ if and only 
if it is surjective for all $\lambda>0$. In such a case, $(0,\infty)\subseteq \rho_{\gamma}(A)$.
\end{enumerate}
\end{prop}
\begin{proof}
Parts (a), (b) and (d) are just \cite[Proposition 3.11, p.~927]{albanese2016} in combination with the sequential completeness 
of $(X,\gamma)$. Part (c) is a consequence of \cite[Chap.~II, 3.14 Proposition (ii), p.~82]{engel_nagel2000} 
and \prettyref{rem:norming_diss_Banach_diss}. 
\end{proof}

In the case $\upsilon=\tau$ parts (a) and (c) of \prettyref{prop:lambda_A} are \cite[Proposition 3.4, p.~6]{budde_wegner2022}.

\begin{defn}
Let $(X,\|\cdot\|)$ be a Banach space. We call a semigroup of linear operators $(T(t))_{t\geq 0}$ from $X$ to $X$ a 
\emph{contraction} semigroup if $\|T(t)\|_{\mathcal{L}(X)}\leq 1$ for all $t\geq 0$.
\end{defn}

\begin{thm}\label{thm:gen_bi_cont_contraction}
Let $(X,\|\cdot\|,\tau)$ be a sequentially complete Saks space, $\upsilon$ a Hausdorff locally convex topology on $X$ with 
$\tau\subseteq\upsilon\subseteq\tau_{\|\cdot\|}$ such that $\gamma$-convergent sequences are $\upsilon$-convergent, $(A,D(A))$ a bi-densely defined, $\Gamma_{\upsilon}$-dissipative operator on $X$ and $\Gamma_{\upsilon}$ norming.
Then the following assertions are equivalent:
\begin{enumerate}
\item $(A,D(A))$ generates a $\tau$-bi-continuous contraction semigroup on $X$. 
\item $\lambda-A$ is surjective for some $\lambda>0$.
\end{enumerate}
\end{thm}
\begin{proof}
We use the Hille--Yosida theorem for bi-continuous semigroups to prove both implications 
(see \cite[Theorem 16, p.~217]{kuehnemund2003} and \cite[Theorem 5.6, p.~340]{budde2019}).  

(a)$\Rightarrow$(b): Let $(A,D(A))$ generate a $\tau$-bi-continuous contraction semigroup on $X$. 
Due to \cite[Theorem 16, p.~217]{kuehnemund2003} with $\omega=0$ we obtain that $(0,\infty)\subseteq\rho(A)$, in particular, 
that $\lambda-A$ is surjective for all $\lambda>0$. 

(b)$\Rightarrow$(a): Let $\lambda-A$ be surjective for some $\lambda>0$. Due to \cite[Theorem 5.6, p.~340]{budde2019} we 
only need to prove that 
\begin{enumerate}
\item[(i)] $(0,\infty)\subseteq\rho(A)$,
\item[(ii)] $\|R(\lambda,A)^{n}\|_{\mathcal{L}(X)}\leq \frac{1}{\lambda^{n}}$ for all $n\in\N$ and all $\lambda>0$, and
\item[(iii)] $\{(\lambda-\alpha)^{n}R(\lambda,A)^{n}\;|\;n\in\N,\,\lambda\geq\alpha\}$ is \emph{bi-equicontinuous} 
for each $\alpha>0$, i.e.~for each $\alpha>0$ and each $\|\cdot\|$-bounded $\tau$-null sequence $(x_{m})_{m\in\N}$ in $X$ 
one has that $\tau$-$\lim_{m\to\infty}(\lambda-\alpha)^{n}R(\lambda,A)^{n}x_{m}=0$ uniformly for all $n\in\N$ 
and all $\lambda\geq\alpha$.
\end{enumerate}
Since $(A,D(A))$ is $\Gamma_{\upsilon}$-dissipative, we get that $\lambda-A$ is bijective for all $\lambda>0$ and 
$\rho(A)\subseteq (0,\infty)$ by \prettyref{prop:lambda_A} (a) and (c). 
From \prettyref{rem:norming_diss_Banach_diss}
and $\Gamma_{\upsilon}$ being norming we deduce that $\|R(\lambda,A)x\|\leq \tfrac{1}{\lambda}\|x\|$ for all $\lambda>0$ 
and $x\in\ran(\lambda-A)=X$, yielding $\|R(\lambda,A)^{n}\|_{\mathcal{L}(X)}\leq \tfrac{1}{\lambda^{n}}$ for all $n\in\N$ 
and $\lambda>0$.
Let $\Gamma_{\tau}$ be a directed system of continuous 
seminorms that generates the topology $\tau$ and $q\in\Gamma_{\tau}$. Thanks to \eqref{eq:res_cont} 
we know that $p((\lambda-A)^{-1}x)\leq \frac{1}{\lambda} p(x)$ for all $p\in\Gamma_{\upsilon}$, $\lambda>0$ and 
$x\in\ran(\lambda-A)=X$. As $\tau\subseteq\upsilon$, there are $p\in\Gamma_{\upsilon}$ and $C\geq 0$ 
such that for each $\alpha>0$ we have
\begin{align*}
    q((\lambda-\alpha)^{n}R(\lambda,A)^{n}x)
&\leq C(\lambda-\alpha)^{n} p((\lambda-A)^{-n}x)
 \leq C\frac{(\lambda-\alpha)^{n}}{\lambda^n} p(x)\\
&\leq C\left(1-\frac{\alpha}{\lambda}\right)^{n}p(x)
 \leq Cp(x)
\end{align*}
for all $x\in X$, $n\in\N$ and $\lambda\geq\alpha$. Since $\|\cdot\|$-bounded $\tau$-null sequences are exactly 
the $\gamma$-null-sequences by \prettyref{rem:sequentially_mixed_bi_cont} (a) and $\gamma$-convergent sequences are assumed 
to be $\upsilon$-convergent, this inequality implies that 
$\{(\lambda-\alpha)^{n}R(\lambda,A)^{n}\;|\;n\in\N,\,\lambda\geq\alpha\}$ is bi-equicontinuous for all $\alpha>0$. 
This finishes the proof.
\end{proof}

In the case $\upsilon=\tau$ we know that $\gamma$-convergent sequences are $\tau$-convergent and thus 
\prettyref{thm:gen_bi_cont_contraction} is \cite[Theorem 3.6, p.~6]{budde_wegner2022} 
(without the superfluous assumption that $A$ should be norm-closed) in this case. 
Another possible choice is $\upsilon=\gamma_{s}$ since the submixed topology $\gamma_s$ has the same convergent 
sequences as $\gamma$ by \prettyref{rem:mixed=submixed} (a). 
However, we are mostly interested in the choice $\upsilon=\gamma$.
Our second generation result involves complete Saks spaces.

\begin{thm}\label{thm:lumer_phillips_mixed_type}
Let $(X,\|\cdot\|,\tau)$ be a complete Saks space and $(A,D(A))$ a $\gamma$-densely defined, $\Gamma_{\gamma}$-dissipative operator. 
Assume that $\ran(\lambda-A)$ is $\gamma$-dense in $X$ for some $\lambda>0$. Then the following assertions hold:
\begin{enumerate}
\item The $\gamma$-closure $(\overline{A},D(\overline{A}))$ generates a $\gamma$-strongly continuous, 
$\gamma$-equicontinu\-ous semigroup $(T(t))_{t\geq 0}$ on $X$. 
\item If $\Gamma_{\gamma}$ is norming, then $(T(t))_{t\geq 0}$ is a contraction semigroup.
\item If $\Gamma_{\gamma}$ is norming and $\gamma=\gamma_s$, then $(T(t))_{t\geq 0}$ is $(\|\cdot\|,\tau)$-equitight.
\end{enumerate}
\end{thm}
\begin{proof}
(a) Due to \cite[Theorem 3.14, p.~929]{albanese2016} $(\overline{A},D(\overline{A}))$ generates a 
$\gamma$-equicontinuous, $\gamma$-strongly continuous semigroup $(T(t))_{t\geq 0}$ on $X$. 

(b) By \cite[Proposition 3.13, p.~929]{albanese2016} the operator $(\overline{A},D(\overline{A}))$ is also 
$\Gamma_{\gamma}$-dissipative and $\lambda-\overline{A}$ is surjective for all $\lambda>0$. 
As a consequence of part (a) and \prettyref{rem:sequentially_mixed_bi_cont} $(T(t))_{t\geq 0}$ is also a 
$\tau$-bi-continuous semigroup on $X$. By \cite[p.~5]{kruse_schwenninger2022a} the generator $(\overline{A},D(\overline{A}))$ 
of $(T(t))_{t\geq 0}$ as a $\gamma$-strongly continuous, $\gamma$-equicontinuous semigroup (see \cite[p.~922]{albanese2016}) 
and the generator of $(T(t))_{t\geq 0}$ as a $\tau$-bi-continuous semigroup (see \prettyref{defn:bi_cont_generator}) coincide. 
Thus $(\overline{A},D(\overline{A}))$ is bi-densely defined by \prettyref{prop:generator} (b).
Hence we get that $(T(t))_{t\geq 0}$ is a contraction semigroup by \prettyref{thm:gen_bi_cont_contraction} with $\upsilon=\gamma$ 
and the norming property of $\Gamma_{\gamma}$. 

(c) It follows from part (b) that $\|T(t)\|_{\mathcal{L}(X)}\leq 1$ for all $t\geq 0$. 
Further, $(T(t))_{t\geq 0}$ is $\gamma$-equicontinuous by part (a). In combination with $\gamma=\gamma_s$ 
we derive that $(T(t))_{t\geq 0}$ is $(\|\cdot\|,\tau)$-equitight by \cite[Proposition 3.16, p.~12--13]{kruse_schwenninger2022}.
\end{proof}

Let us compare \prettyref{thm:lumer_phillips_mixed_type} with one of the main theorems of \cite{budde_wegner2022}, namely, 
\cite[Theorem 3.15, p.~11]{budde_wegner2022}. 
We note that the topology that is called mixed topology (and denoted by $\gamma$ there) in \cite[p.~10]{budde_wegner2022} is 
actually the submixed topology $\gamma_s$. With this observation at hand let us phrase \cite[Theorem 3.15, p.~11]{budde_wegner2022} 
in our terminology. 

\begin{thm}[{\cite[Theorem 3.15, p.~11]{budde_wegner2022}}]\label{thm:lumer_phillips_budde_wegner}
Let $(X,\|\cdot\|,\tau)$ be a sequentially complete Saks space such that $(X,\gamma_{s})$ is complete, 
and $(A,D(A))$ a bi-densely defined, bi-dissipative operator. 
Assume that $\ran(\lambda-A)$ is bi-dense, i.e.~sequentially $\gamma$-dense, in $X$ for some $\lambda>0$. 
Then the $\gamma_{s}$-closure $(\overline{A}^{\gamma_s},D(\overline{A}^{\gamma_s}))$ generates 
a $\tau$-bi-continuous contraction semigroup on $X$.
\end{thm}

\begin{rem}\label{rem:gap_budde_wegner}
\begin{enumerate}
\item First, it is actually shown in the proof of \prettyref{thm:lumer_phillips_budde_wegner} 
that $(\overline{A}^{\gamma_s},D(\overline{A}^{\gamma_s}))$ generates a $\gamma_{s}$-strongly continuous, 
$\gamma_{s}$-equicontinuous semigroup on $X$, in particular, a $\tau$-bi-continuous semigroup by \prettyref{rem:sequentially_mixed_bi_cont} and \prettyref{rem:mixed=submixed} (a). 
However, the proof that the generated semigroup is contractive is missing. 
In this proof it is used that a bi-dissipative operator is 
$(\vertiii{\cdot}_{(p_n,a_n)_{n\in\N}})_{(p_n,a_n)_{n\in\N}\in\mathcal{N}}$-dissipative 
as well (see \prettyref{rem:bi_dissip_submixed}).
In order to prove that the generated semigroup is also contractive, the only available tool in \cite{budde_wegner2022}
is \cite[Theorem 3.6, p.~6]{budde_wegner2022}. However, to apply the latter theorem one has to show that 
$(\overline{A}^{\gamma_s},D(\overline{A}^{\gamma_s}))$ is also bi-dissipative. 
Due to \cite[Proposition 3.13, p.~929]{albanese2016} we only know that $(\overline{A}^{\gamma_s},D(\overline{A}^{\gamma_s}))$ is 
$(\vertiii{\cdot}_{(p_n,a_n)_{n\in\N}})_{(p_n,a_n)_{n\in\N}\in\mathcal{N}}$-dissipative 
(and $\lambda-\overline{A}^{\gamma_s}$ is surjective for all $\lambda>0$). 
To circumvent this obstacle, we relaxed \cite[Theorem 3.6, p.~6]{budde_wegner2022} to 
\prettyref{thm:gen_bi_cont_contraction} where one has several possible choices for the topology $\upsilon$, not only 
$\upsilon=\tau$ as in \cite[Theorem 3.6, p.~6]{budde_wegner2022}. Using \prettyref{rem:bi_dissip_submixed} and 
\cite[Proposition 3.13, p.~929]{albanese2016}, we see 
that $(\overline{A}^{\gamma_s},D(\overline{A}^{\gamma_s}))$ is $\Gamma_{\gamma_s}$-dissipative w.r.t~the norming directed system of 
continuous seminorms $\Gamma_{\gamma_s}$ from \prettyref{rem:submixed_norming}. 
Now, it is possible to apply \prettyref{thm:gen_bi_cont_contraction} 
with $\upsilon=\gamma_s$ to conclude that the generated semigroup is contractive. This closes the gap in the proof of
\prettyref{thm:lumer_phillips_budde_wegner}. 
\item There is no nice characterisation (known to us) of the completeness of $(X,\gamma_{s})$, that is assumed in \prettyref{thm:lumer_phillips_budde_wegner}. However, there is a nice characterisation of the completeness of the Saks space 
$(X,\|\cdot\|,\tau)$. By definition the Saks space is complete if and only if 
$(X,\gamma)$ is complete. The space $(X,\gamma)$ is complete if and only if $B_{\|\cdot\|}=\{x\in X\;|\; \|x\|\leq 1\}$ 
is $\tau$-complete by \cite[I.1.14 Proposition, p.~11]{cooper1978}. But, since $\gamma_s$ is in general a weaker topology than 
$\gamma$ by \prettyref{rem:mixed=submixed} (a), the completeness of $(X,\gamma)$ does in general not imply the completeness of 
$(X,\gamma_s)$.  
\item Let us suppose that $\gamma=\gamma_s$. Then \prettyref{thm:lumer_phillips_budde_wegner} 
is covered by \prettyref{thm:lumer_phillips_mixed_type} (a) and (b). Further, we point out that 
in comparison to \prettyref{thm:lumer_phillips_budde_wegner} 
we weakened the assumptions from $(A,D(A))$ being a bi-densely defined, bi-dissipative operator 
and $\ran(\lambda-A)$ being bi-dense for some $\lambda>0$ 
to $(A,D(A))$ being a $\gamma$-densely defined, $\Gamma_{\gamma}$-dissipative operator 
and $\ran(\lambda-A)$ being $\gamma$-dense for some $\lambda>0$ (see \prettyref{rem:mixed_closed_densely_defined}) 
in \prettyref{thm:lumer_phillips_mixed_type}. 
\end{enumerate}
\end{rem}

Let us take a closer look at the completeness assumption on the Saks space $(X,\|\cdot\|,\tau)$ in 
\prettyref{thm:lumer_phillips_mixed_type}, which is actually fulfilled for many important examples, 
and its characterisation in \prettyref{rem:gap_budde_wegner} (b). 
Especially, $(X,\gamma)$ is complete, thus $(X,\|\cdot\|,\tau)$ as well, 
if $B_{\|\cdot\|}$ is $\tau$-compact, which is condition (ii) of \prettyref{rem:mixed=submixed} (b) and also a 
sufficient condition for $\gamma=\gamma_{s}$.
We recall the following observations from \cite[Examples 2.4, 3.11, p.~4--5, 10]{kruse_schwenninger2022}, 
\cite[Remark 3.20 (a), p.~15]{kruse_schwenninger2022}, \cite[Example 4.12, p.~24--25]{kruse_schwenninger2022} and 
\cite[Corollary 3.23, p.~17]{kruse_schwenninger2022}, and add a proof of the completeness of the Saks spaces considered 
in \prettyref{rem:complete_saks} (c), (d) and (f) below.

\begin{rem}\label{rem:complete_saks}
\begin{enumerate}
\item Let $\Omega$ be a Hausdorff $k_{\R}$-space and recall that a completely regular space $\Omega$ is called \emph{$k_{\R}$-space} 
if any map $f\colon\Omega\to\R$ whose restriction to each compact $K\subset\Omega$ is continuous, is already continuous on $\Omega$ 
(see \cite[p.~487]{michael1973}). 
Further, let $\mathrm{C}_{\operatorname{b}}(\Omega)$ be the space of bounded continuous 
functions on $\Omega$, and $\|\cdot\|_{\infty}$ the sup-norm as well as $\tau_{\operatorname{co}}$ the compact-open topology, i.e.~the topology of uniform convergence on compact subsets of $\Omega$. 
Then $(\mathrm{C}_{\operatorname{b}}(\Omega),\|\cdot\|_{\infty},\tau_{\operatorname{co}})$ is a complete Saks space 
and $\gamma(\|\cdot\|_{\infty},\tau_{\operatorname{co}})=\gamma_{s}(\|\cdot\|_{\infty},\tau_{\operatorname{co}})$.

Let $\mathcal{V}$ denote the set of all non-negative bounded functions $\nu$ on $\Omega$ 
that vanish at infinity, i.e.~for every $\varepsilon>0$ the set $\{x\in\Omega\;|\;\nu(x)\geq\varepsilon\}$ is compact. 
Let $\beta_{0}$ be the Hausdorff locally convex topology on $\mathrm{C}_{\operatorname{b}}(\Omega)$ that is induced 
by the seminorms 
\[
|f|_{\nu}\coloneqq\sup_{x\in\Omega}|f(x)|\nu(x),\quad f\in\mathrm{C}_{\operatorname{b}}(\Omega),
\]
for $\nu\in\mathcal{V}$. Then we have $\gamma(\|\cdot\|_{\infty},\tau_{\operatorname{co}})=\beta_{0}$. 
If $\Omega$ is locally compact, then $\mathcal{V}$ may be replaced by the functions 
in $\mathrm{C}_{0}(\Omega)$ that are non-negative where 
$\mathrm{C}_{0}(\Omega)$ is the space of real-valued continuous functions on $\Omega$ that vanish at infinity.

If $\Omega$ is a hemicompact Hausdorff $k_{\R}$-space 
or a Polish space, then we even have
\[
 \gamma(\|\cdot\|_{\infty},\tau_{\operatorname{co}})
=\beta_{0}=\mu(\mathrm{C}_{\operatorname{b}}(\Omega),\mathrm{M}_{\operatorname{t}}(\Omega))
\]
where $\mathrm{M}_{\operatorname{t}}(\Omega)=(\mathrm{C}_{\operatorname{b}}(\Omega),\beta_{0})'$ 
is the space of bounded Radon measures and $\mu(\mathrm{C}_{\operatorname{b}}(\Omega),\mathrm{M}_{\operatorname{t}}(\Omega))$ 
the Mackey-topology of the dual pair $(\mathrm{C}_{\operatorname{b}}(\Omega),\mathrm{M}_{\operatorname{t}}(\Omega))$.

\item Let $(X,\|\cdot\|)$ be a Banach space and $\sigma^{\ast}\coloneqq \sigma(X',X)$ the weak$^{\ast}$-topology.
Then condition (ii) of \prettyref{rem:mixed=submixed} (b) is fulfilled, 
$(X',\|\cdot\|_{X'},\sigma^{\ast})$ is a complete Saks space and 
$\gamma(\|\cdot\|_{X'},\sigma^{\ast})=\gamma_{s}(\|\cdot\|_{\infty},\sigma^{\ast})=\tau_{\operatorname{c}}(X',X)$ 
where $\tau_{\operatorname{c}}(X',X)$ is the topology of uniform convergence on compact subsets of $X$.
\item Let $(X,\|\cdot\|)$ be a Banach space and $\mu^{\ast}\coloneqq\mu(X',X)$ the dual Mackey-topology.
Then $(X',\|\cdot\|_{X'},\mu^{\ast})$ is a complete Saks space, where the completeness follows from 
\cite[p.~74]{kalton1973}, and $\gamma(\|\cdot\|_{X'},\mu^{\ast})=\mu^{\ast}$. 
If $X$ is a \emph{Schur space}, i.e.~every $\sigma(X,X')$-convergent sequence is $\|\cdot\|$-convergent 
(see \cite[p.~253]{fabian2011}), then condition (ii) of \prettyref{rem:mixed=submixed} (b) is fulfilled and 
$\gamma(\|\cdot\|_{X'},\mu^{\ast})=\gamma_{s}(\|\cdot\|_{X'},\mu^{\ast})$.
\item Let $(X,\|\cdot\|_{X})$ and $(Y,\|\cdot\|_{Y})$ be Banach spaces, and $\tau_{\operatorname{sot}}$
the strong operator topology on $\mathcal{L}(X;Y)$. 
Then $(\mathcal{L}(X;Y),\|\cdot\|_{\mathcal{L}(X;Y)},\tau_{\operatorname{sot}})$ is a Saks space. 
Let $(T_{i})_{i\in I}$ be a $\tau_{\operatorname{sot}}$-Cauchy net in 
$B_{\|\cdot\|_{\mathcal{L}(X;Y)}}=\{T\in\mathcal{L}(X;Y)\;|\; \|T\|_{\mathcal{L}(X;Y)}\leq 1\}$. 
Then for each $x\in X$ the net $(T_{i}x)_{i\in I}$ is $\|\cdot\|_{Y}$-convergent to some $Tx\in Y$ with 
$\|Tx\|_{Y}\leq \|x\|_{X}$ in the Banach space $(Y,\|\cdot\|_{Y})$. Thus the map $T\colon x\mapsto Tx$ belongs to 
$\mathcal{L}(X;Y)$ with $\|T\|_{\mathcal{L}(X;Y)}\leq 1$ and $(T_{i})_{i\in I}$ is $\tau_{\operatorname{sot}}$-convergent to $T$. 
Hence $B_{\|\cdot\|_{\mathcal{L}(X;Y)}}$ is $\tau_{\operatorname{sot}}$-complete and 
so $(\mathcal{L}(X;Y),\|\cdot\|_{\mathcal{L}(X;Y)},\tau_{\operatorname{sot}})$ is complete.
If $Y$ is in addition finite-dimensional, then condition (ii) of \prettyref{rem:mixed=submixed} (b) is fulfilled and 
$\gamma(\|\cdot\|_{\mathcal{L}(X;Y)},\tau_{\operatorname{sot}})=\gamma_{s}(\|\cdot\|_{\mathcal{L}(X;Y)},\tau_{\operatorname{sot}})$.
\item Let $H$ be a separable Hilbert space and $\mathcal{N}(H)$ the space of trace class operators 
in $\mathcal{L}(H)$ and note that $\mathcal{L}(H)=\mathcal{N}(H)'$.  
Let $\tau_{\operatorname{sot}^{\ast}}$ be the symmetric
strong operator topology, i.e.~the Hausdorff locally convex topology on $\mathcal{L}(H)$ 
generated by the directed system of seminorms
\[
p_{N}(R)\coloneqq\max\bigl(\sup_{x\in N}\|Rx\|_{H},\sup_{x\in N}\|R^{\ast}x\|_{H}\bigr),\quad R\in \mathcal{L}(H),
\]
for finite $N\subset H$ where $R^{\ast}$ is the adjoint of $R$. 
We denote by $\beta_{\operatorname{sot}^{\ast}}$ the mixed topology 
$\gamma(\|\cdot\|_{\mathcal{L}(H)},\tau_{\operatorname{sot}^{\ast}})$. 
Then the triple $(\mathcal{L}(H),\|\cdot\|_{\mathcal{L}(H)},\tau_{\operatorname{sot}^{\ast}})$ 
is a complete Saks space and $\beta_{\operatorname{sot}^{\ast}}=\mu(\mathcal{L}(H),\mathcal{N}(H))$. 
\item Let $\Omega$ be a completely regular Hausdorff space, $\mathrm{M}_{\operatorname{t}}(\Omega)$ the space of bounded Radon 
measures on $\Omega$, and $\|\cdot\|_{\mathrm{M}_{\operatorname{t}}(\Omega)}$ the total variation norm on 
$\mathrm{M}_{\operatorname{t}}(\Omega)$. Then $(\mathrm{M}_{\operatorname{t}}(\Omega),\|\cdot\|_{\mathrm{M}_{\operatorname{t}}(\Omega)},\sigma(\mathrm{M}_{\operatorname{t}}(\Omega),\mathrm{C}_{\operatorname{b}}(\Omega)))$ is a complete Saks space 
where the completeness follows from $B_{\|\cdot\|_{\mathrm{M}_{\operatorname{t}}(\Omega)}}$ being 
$\sigma(\mathrm{M}_{\operatorname{t}}(\Omega),\mathrm{C}_{\operatorname{b}}(\Omega))$-compact
by \cite[Corollary 3.23 (a), p.~17]{kruse_schwenninger2022}, 
i.e.~condition (ii) of \prettyref{rem:mixed=submixed} (b) is fulfilled. Furthermore, we have 
\begin{align*}
\beta_{0}'\coloneqq &\gamma(\|\cdot\|_{\mathrm{M}_{\operatorname{t}}(\Omega)},\sigma(\mathrm{M}_{\operatorname{t}}(\Omega),\mathrm{C}_{\operatorname{b}}(\Omega)))
=\gamma_{s}(\|\cdot\|_{\mathrm{M}_{\operatorname{t}}(\Omega)},\sigma(\mathrm{M}_{\operatorname{t}}(\Omega),\mathrm{C}_{\operatorname{b}}(\Omega)))\\
=&\tau_{\operatorname{c}}(\mathrm{M}_{\operatorname{t}}(\Omega),(\mathrm{C}_{\operatorname{b}}(\Omega),\|\cdot\|_{\infty})).
\end{align*}
\end{enumerate}
\end{rem}

Let us consider a toy example for an application of \prettyref{thm:lumer_phillips_mixed_type}, namely, 
the multiplication operator on $\mathrm{C}_{\operatorname{b}}(\Omega)$, which we will revisit for other generation results.

\begin{exa}\label{ex:multiplication_complete}
Let $\Omega$ be a Hausdorff $k_{\R}$-space and $q\colon\Omega\to\C$ be continuous with 
$C\coloneqq\sup_{x\in\Omega}\re q(x)<\infty$. We define the multiplication operator $(M_{q},D(M_{q}))$ by setting 
\[
D(M_{q})\coloneqq\{f\in\mathrm{C}_{\operatorname{b}}(\Omega)\;|\;qf\in\mathrm{C}_{\operatorname{b}}(\Omega)\}
\]
and $M_{q}\coloneqq qf$ for $f\in D(M_{q})$. By solving the equation $(\lambda-q)f=g$ we can compute 
the resolvent $R(\lambda,M_{q})$ of $M_{q}$ explicitly by 
\[
R(\lambda,M_{q})f=\frac{1}{\lambda-q}f,\quad f\in \mathrm{C}_{\operatorname{b}}(\Omega),
\]
for all $\lambda\in(\C\setminus\overline{q(\Omega)})=\rho(M_{q})$, which shows that $\lambda-M_{q}$ is surjective, 
i.e.~$\ran(\lambda-M_{q})=\mathrm{C}_{\operatorname{b}}(\Omega)$, for all $\lambda\in\C\setminus\overline{q(\Omega)}$. 
Suppose that $C\leq 0$. Then $(0,\infty)\in \rho(M_{q})$ and $\ran(\lambda-M_{q})=\mathrm{C}_{\operatorname{b}}(\Omega)$ 
for all $\lambda>0$. Furthermore, we have for all $\lambda>0$, $f\in\mathrm{C}_{\operatorname{b}}(\Omega)$ and 
$\nu\in\mathcal{V}$ from \prettyref{rem:complete_saks} (a) that 
\begin{align*}
  |R(\lambda,M_{q})f|_{\nu}
&=\sup_{x\in\Omega}\frac{1}{|\lambda-q(x)|}|f(x)|\nu(x)
 \underset{C\leq 0}{\leq}\sup_{x\in\Omega}\frac{1}{\lambda-\re q(x)}|f(x)|\nu(x)\\
&\leq\frac{1}{\lambda}\sup_{x\in\Omega}|f(x)|\nu(x) 
 =\frac{1}{\lambda}|f|_{\nu}.
\end{align*}
Therefore $(M_{q},D(M_{q}))$ is $\Gamma_{\beta_{0}}$-dissipative for the directed system of seminorms 
$\Gamma_{\beta_{0}}\coloneqq (|\cdot|_{\nu})_{\nu\in\mathcal{V}}$ that generates the mixed topology 
$\beta_{0}=\gamma(\|\cdot\|_{\infty},\tau_{\operatorname{co}})$. Moreover, due to \prettyref{prop:lambda_A} (b)
and $\ran(\lambda-M_{q})=\mathrm{C}_{\operatorname{b}}(\Omega)$ 
for all $\lambda>0$ the operator $(M_{q},D(M_{q}))$ is $\beta_{0}$-closed and thus 
generates a $\beta_{0}$-strongly continuous, $\beta_{0}$-equicontinuous semigroup $(T(t))_{t\geq 0}$ 
on $\mathrm{C}_{\operatorname{b}}(\Omega)$ by \prettyref{thm:lumer_phillips_mixed_type} (a) and \prettyref{rem:complete_saks} (a). 
Choosing $\mathcal{V}_{1}\coloneqq\{\nu\in\mathcal{V}\;|\;\forall\;x\in\Omega:\;\nu(x)\leq 1\}$ instead of $\mathcal{V}$, 
we get a norming directed system of continuous seminorms that generates $\beta_{0}$ for which $(M_{q},D(M_{q}))$ is dissipative, too. 
Hence $(T(t))_{t\geq 0}$ is also a $(\|\cdot\|_{\infty},\tau_{\operatorname{co}})$-equitight contraction semigroup 
by \prettyref{thm:lumer_phillips_mixed_type} (b) and (c) since 
$\beta_{0}=\gamma(\|\cdot\|_{\infty},\tau_{\operatorname{co}})=\gamma_{s}(\|\cdot\|_{\infty},\tau_{\operatorname{co}})$ 
by \prettyref{rem:complete_saks} (a).
\end{exa}

Our next generation result involves the $\gamma$-dual operator. 
Let us recall some observations from \cite[Remark 4.5, p.~9]{kruse_schwenninger2022a}. 
Let $(X,\|\cdot\|,\tau)$ be a sequentially complete Saks space. 
Then $X_{\gamma}'\coloneqq (X,\gamma)'$ is a closed linear subspace of $X'$, in particular a Banach space, 
by \cite[I.1.18 Proposition, p.~15]{cooper1978}, and we denote by $\|\cdot\|_{X_{\gamma}'}$ the restriction of 
$\|\cdot\|_{X'}$ to $X_{\gamma}'$. We note that $(X,\gamma)$ is a \emph{Mazur space}, 
i.e.~$X_{\gamma}'$ coincides with the space of linear $\gamma$-sequentially continuous functionals on $X$ 
(see \cite[p.~50]{wilansky1981}), if and only if 
\[
X_{\gamma}'=\{x'\in X'\;|\;x'\;\tau\text{-sequentially continuous on } \|\cdot\|\text{-bounded sets}\}\eqqcolon X^{\circ}
\]
by \cite[I.1.10 Proposition, p.~9]{cooper1978}. 
The space $X^{\circ}$ was introduced in \cite[Proposition 2.1, p.~314]{farkas2011} in the context of dual semigroups 
of bi-continuous semigroups. 

\begin{cor}\label{cor:lumer_phillips_dual}
Let $(X,\|\cdot\|,\tau)$ be a complete Saks space. Let both $(A,D(A))$ and its $\gamma$-dual operator $(A',D(A'))$ be 
$\Gamma_{\gamma}$-dissipative and $\|\cdot\|_{X_{\gamma}'}$-dissipative operators, respectively.
Then the following assertions hold:
\begin{enumerate}
\item The $\gamma$-closure $(\overline{A},D(\overline{A}))$ generates a $\gamma$-equicontinuous, 
$\gamma$-strongly continuous semigroup $(T(t))_{t\geq 0}$ on $X$. 
\item If $\Gamma_{\gamma}$ is norming, then $(T(t))_{t\geq 0}$ is a contraction semigroup.
\item If $\Gamma_{\gamma}$ is norming and $\gamma=\gamma_s$, then $(T(t))_{t\geq 0}$ is $(\|\cdot\|,\tau)$-equitight.
\end{enumerate}  
\end{cor}
\begin{proof}
By \cite[I.1.18 Proposition (i), p.~15]{cooper1978} we have $(X_{\gamma}',\tau_{b})=(X_{\gamma}',\|\cdot\|_{X_{\gamma}'})$ 
where $\tau_{b}$ denotes the topology of uniform convergence on $\gamma$-bounded sets.
Due to \cite[Corollary 3.17, p.~931]{albanese2016} we get that $(\overline{A},D(\overline{A}))$ generates a 
$\gamma$-strongly continuous, $\gamma$-equicontinuous semigroup on $X$. 
Parts (b) and (c) follow as in \prettyref{thm:lumer_phillips_mixed_type}.
\end{proof}

\begin{exa}\label{ex:multiplication_dual}
Let $\Omega\coloneqq\N$ be equipped with the metric induced by the absolute value. Then $\Omega$ is a Polish space, 
in particular, a Hausdorff $k_{\R}$-space. Moreover, $\mathrm{C}_{\operatorname{b}}(\N)=\ell^{\infty}$ 
and $\mathrm{M}_{\operatorname{t}}(\N)=\ell^{1}$ (see e.g.~\cite[p.~477]{conway1967}). 
It follows from \prettyref{rem:complete_saks} (a) that 
$\beta_{0}=\mu(\ell^{\infty},\ell^{1})$ and so 
\[
(\ell^{\infty},\beta_{0})'=(\ell^{\infty},\mu(\ell^{\infty},\ell^{1}))'=\ell^{1}.
\]
Let $q\colon\N\to\C$ be a function with $C\coloneqq\sup_{n\in\N}\re q(n)\leq 0$. 
Again, we consider the multiplication operator $M_{q}$ from \prettyref{ex:multiplication_complete}, i.e.
\[
D(M_{q})=\{f\in\ell^{\infty}\;|\;qf\in\ell^{\infty}\}
\]
and $M_{q}= qf$ for $f\in D(M_{q})$. We already know that $(M_{q},D(M_{q}))$ is $\Gamma_{\beta_{0}}$-dissipative 
with $\Gamma_{\beta_{0}}$ from \prettyref{ex:multiplication_complete}. Furthermore, we have for the $\beta_{0}$-dual 
operator $(M_{q}',D(M_{q}'))$ that 
\[
D(M_{q}')=\{f\in\ell^{1}\;|\;qf\in\ell^{1}\}
\]
and $M_{q}'=qf$ for $f\in D(M_{q}')$. For all $\lambda>0$ and $f\in D(M_{q}')$ we get 
\[
 \|(\lambda-M_{q}')f\|_{\ell^{1}}
=\sum_{n=1}^{\infty}|(\lambda-q(n))f_{n}|
\underset{C\leq 0}{\geq}\sum_{n=1}^{\infty}(\lambda-\re q(n))|f_{n}|
\geq \lambda \sum_{n=1}^{\infty}|f_{n}|
=\lambda\|f\|_{\ell^{1}},
\]
meaning that $(M_{q}',D(M_{q}'))$ is $\|\cdot\|_{\ell^{1}}$-dissipative. Thus we may apply 
\prettyref{cor:lumer_phillips_dual} (a) to deduce that $(M_{q},D(M_{q}))$ generates a 
$\mu(\ell^{\infty},\ell^{1})$-strongly continuous, $\mu(\ell^{\infty},\ell^{1})$-equicontinuous semigroup on $\ell^{\infty}$.
\end{exa}

Instead of \prettyref{rem:complete_saks} (a) we may also use \prettyref{rem:complete_saks} (c) in 
\prettyref{ex:multiplication_dual} since $\ell^{1}$ is a Schur space by \cite[Theorem 5.36, p.~252]{fabian2011}.

Next, we would like to transfer \cite[Theorem 3.18, p.~931]{albanese2016} to the setting of Saks spaces $(X,\|\cdot\|,\tau)$. 
However, looking at the assumptions of \cite[Theorem 3.18, p.~931]{albanese2016}, we see that this requires $(X,\gamma)$ to 
be reflexive. Since reflexive spaces are barrelled, this requirement implies that $\tau=\gamma=\tau_{\|\cdot\|}$ 
by \cite[I.1.15 Proposition, p.~12]{cooper1978} and so we are in an uninteresting situation from the perspective of 
$\tau$-bi-continuous semigroups. But if we could relax the assumption to $(X,\gamma)$ being semi-reflexive, then 
there are non-trivial (i.e.~not Banach) Saks spaces. This is actually possible by the following observation.

\begin{rem}\label{rem:semireflexive}
\cite[Theorem 3.18, p.~931]{albanese2016} is stated for reflexive Hausdorff locally convex spaces $(X,\upsilon)$. 
However, a closer look at its proof reveals that it is actually valid for semi-reflexive $(X,\upsilon)$ because the only part 
where reflexivity comes into play is that it implies that a $\upsilon$-bounded set $B\subset X$ is relatively 
$\sigma(X,(X,\upsilon)')$-compact; see \cite[p.~931, l.~9--10 from below]{albanese2016}. 
But the latter assertion is equivalent to semi-reflexivity by \cite[Proposition 23.18, p.~270]{meisevogt1997}.
\end{rem}

\begin{thm}\label{thm:lumer_phillips_semi_reflexive}
Let $(X,\|\cdot\|,\tau)$ be a complete, semi-reflexive Saks space, $(A,D(A))$ a $\gamma$-densely defined, 
$\Gamma_{\gamma}$-dissipative operator and $\ran(\lambda-A)=X$ for some $\lambda>0$.
Then the following assertions hold:
\begin{enumerate}
\item $(A,D(A))$ generates a $\gamma$-equicontinuous, $\gamma$-strongly continuous semigroup $(T(t))_{t\geq 0}$ on $X$. 
\item If $\Gamma_{\gamma}$ is norming, then $(T(t))_{t\geq 0}$ is a contraction semigroup.
\item If $\Gamma_{\gamma}$ is norming and $\gamma=\gamma_s$, then $(T(t))_{t\geq 0}$ is $(\|\cdot\|,\tau)$-equitight.
\end{enumerate}  
\end{thm}
\begin{proof}
Part (a) follows from \cite[Theorem 3.18, p.~931]{albanese2016} (noting that $1-A$ is surjective by \prettyref{prop:lambda_A} (d)) and \prettyref{rem:semireflexive}. 
Parts (b) and (c) follow as in \prettyref{thm:lumer_phillips_mixed_type}.
\end{proof}

Let $(X,\|\cdot\|,\tau)$ be a Saks space. By definition the Saks space is semi-reflexive if and only if 
$(X,\gamma)$ is semi-reflexive.
The space $(X,\gamma)$ is semi-reflexive if and only if $B_{\|\cdot\|}=\{x\in X\;|\; \|x\|\leq 1\}$ 
is $\sigma(X,(X,\tau)')$-compact by \cite[I.1.21 Corollary, p.~16]{cooper1978}. Due to 
\cite[I.1.20 Proposition, p.~16]{cooper1978}, $B_{\|\cdot\|}$ is $\sigma(X,(X,\tau)')$-compact if and only if 
it is $\sigma(X,(X,\gamma)')$-compact. Further, $(X,\gamma)$ is a semi-Montel space, 
thus semi-reflexive, if and only if $B_{\|\cdot\|}$ is $\tau$-compact by 
\cite[I.1.13 Proposition, p.~11]{cooper1978} which is condition (ii) in \prettyref{rem:mixed=submixed} (b) again 
and also a sufficient condition for $\gamma=\gamma_{s}$. Therefore we have by \prettyref{rem:complete_saks} 
the following observations where we only have to add an additional argument in parts (a), (c) and (e) of 
\prettyref{rem:semi_reflexive_saks} below.

\begin{rem}\label{rem:semi_reflexive_saks}  
\begin{enumerate}
\item Let $\Omega$ be a discrete space. 
Then $(\mathrm{C}_{\operatorname{b}}(\Omega),\|\cdot\|_{\infty},\tau_{\operatorname{co}})$ is a complete, semi-reflexive Saks space 
by \cite[II.1.24 Remark 4), p.~88--89]{cooper1978}.
\item Let $(X,\|\cdot\|)$ be a Banach space.
Then $(X',\|\cdot\|_{X'},\sigma^{\ast})$ is a complete, semi-reflexive Saks space.
\item Let $(X,\|\cdot\|)$ be a Banach space. Then $(X',\|\cdot\|_{X'},\mu^{\ast})$ is a complete, semi-reflexive Saks space 
where the semi-reflexivity follows from $(X',\mu^{\ast})''=X'$ by the Mackey--Arens theorem.
\item Let $(X,\|\cdot\|_{X})$ and $(Y,\|\cdot\|_{Y})$ be Banach spaces and $Y$ finite-dimensional.
Then $(\mathcal{L}(X;Y),\|\cdot\|_{\mathcal{L}(X;Y)},\tau_{\operatorname{sot}})$ is a complete, semi-reflexive Saks space.
\item Let $H$ be a separable Hilbert space. 
Then $(\mathcal{L}(H),\|\cdot\|_{\mathcal{L}(H)},\tau_{\operatorname{sot}^{\ast}})$ 
is a complete, semi-reflexive Saks space where the semi-reflexivity follows from 
$(\mathcal{L}(H),\beta_{\operatorname{sot}^{\ast}})''=\mathcal{N}(H)'=\mathcal{L}(H)$. 
\item Let $\Omega$ be a completely regular Hausdorff space. 
Then we have that the triple $(\mathrm{M}_{\operatorname{t}}(\Omega),\|\cdot\|_{\mathrm{M}_{\operatorname{t}}(\Omega)},\sigma(\mathrm{M}_{\operatorname{t}}(\Omega),\mathrm{C}_{\operatorname{b}}(\Omega)))$ is a complete, semi-reflexive Saks space.
\end{enumerate}
\end{rem}

\begin{exa}\label{ex:multiplication_semi_reflexive} 
Due to \prettyref{ex:multiplication_dual} and \prettyref{rem:semi_reflexive_saks} (a) 
$(\ell^{\infty},\|\cdot\|_{\infty},\tau_{\operatorname{co}})$ is a complete, semi-reflexive Saks space. 
Therefore we may also apply \prettyref{thm:lumer_phillips_semi_reflexive} (a) 
to prove that the multiplication operator $(M_{q},D(M_{q}))$ with $\sup_{n\in\N}\re q(n)\leq 0$ generates a 
$\mu(\ell^{\infty},\ell^{1})$-strongly continuous, $\mu(\ell^{\infty},\ell^{1})$-equicontinuous semigroup on $\ell^{\infty}$ 
(we already checked in \prettyref{ex:multiplication_complete} that the other assumptions of 
\prettyref{thm:lumer_phillips_semi_reflexive} are satisfied).
\end{exa}

We close this section with a characterisation of the bi-continuous semigroups with dissipative generators. 
First, we start with a refinement of \prettyref{thm:gen_bi_cont_contraction}. 
In the case $\upsilon=\tau$ this was already done in \cite[Proposition 3.11, p.~9]{budde_wegner2022} whose prove needs 
some adaptations in the case of more general Hausdorff locally convex topologies $\upsilon$ with 
$\tau\subseteq\upsilon\subseteq\tau_{\|\cdot\|}$ for sequentially complete Saks spaces $(X,\|\cdot\|,\tau)$.

\begin{prop}\label{prop:gen_bi_cont_contraction}
Let $(X,\|\cdot\|,\tau)$ be a sequentially complete Saks space, $\upsilon$ a Hausdorff locally convex topology on $X$ with 
$\tau\subseteq\upsilon\subseteq\tau_{\|\cdot\|}$ such that $\gamma$-convergent sequences are $\upsilon$-convergent, and 
$(A,D(A))$ bi-densely defined.
Then the following assertions are equivalent:
\begin{enumerate}
\item $(A,D(A))$ generates a $\tau$-bi-continuous contraction semigroup $(T(t))_{t\geq 0}$ on $X$ and there exists a norming
directed system of continuous seminorms $\Gamma_{\upsilon}$ that generates $\upsilon$ such that 
$p(T(t)x)\leq p(x)$ for all $t\geq 0$, $p\in\Gamma_{\upsilon}$ and $x\in X$.
\item $\lambda-A$ is surjective for some $\lambda>0$ and $(A,D(A))$ is a $\Gamma_{\upsilon}$-dissipative operator on $X$ for some
norming directed system of continuous seminorms $\Gamma_{\upsilon}$ that generates $\upsilon$.
\end{enumerate}
\end{prop}
\begin{proof}
(a)$\Rightarrow$(b): First, we show that $(A,D(A))$ is $\Gamma_{\upsilon}$-dissipative. 
We note that $(0,\infty)\subseteq\rho(A)$ and
\[
R(\lambda,A)x=\int_{0}^{\infty}\euler^{-\lambda t}T(t)x\d t
\]
for all $\lambda>0$ and $x\in X$ by \cite[Theorem 12, p.~215]{kuehnemund2003} and \cite[Definition 9, p.~213]{kuehnemund2003} 
where the integral is an improper $\tau$-Riemann integral. The sequence of Riemann sums that approximate the integral 
on the right-hand side w.r.t.~$\tau$ are $\|\cdot\|$-bounded for each $\lambda>0$ and $x\in X$. 
Due to \prettyref{rem:sequentially_mixed_bi_cont} (a) this means that this sequence of Riemann sums is actually $\gamma$-convergent 
and thus $\upsilon$-convergent by assumption. Therefore we have for all $\lambda>0$, $p\in\Gamma_{\upsilon}$ and $x\in X$ that 
\[
    p\left(R(\lambda,A)x\right)
\leq\int_{0}^{\infty}\euler^{-\lambda t}p(T(t)x)\d t
\leq\int_{0}^{\infty}\euler^{-\lambda t}p(x)\d t
=\frac{1}{\lambda}p(x)
\] 
where we used that $p$ is $\upsilon$-continuous for the first inequality. Hence $(A,D(A))$ is $\Gamma_{\upsilon}$-dissipative. 
In combination with \prettyref{thm:gen_bi_cont_contraction} this yields that $\lambda-A$ is surjective for some $\lambda>0$. 

(b)$\Rightarrow$(a): Due to \prettyref{thm:gen_bi_cont_contraction}, $(A,D(A))$ generates a 
$\tau$-bi-continuous contraction semigroup on $X$, and $(0,\infty)\subseteq\rho(A)$ by \prettyref{prop:lambda_A} (c). 
Furthermore, we have by the Post--Widder inversion formula \cite[Corollary 2.10, p.~47]{kuehnemund2001} that 
\[
T(t)x=\tau\text{-}\lim_{n\to\infty}\left(\frac{n}{t}R\left(\frac{n}{t},A\right)\right)^{n}x
\]
for all $t>0$ and $x\in X$. As a consequence of \prettyref{rem:norming_diss_Banach_diss} the $\tau$-convergent sequence 
$((\tfrac{n}{t}R(\tfrac{n}{t},A))^{n}x)_{n\in\N}$ is $\|\cdot\|$-bounded for each $t>0$ and $x\in X$, 
thus $\gamma$-convergent by \prettyref{rem:sequentially_mixed_bi_cont} (a) and so $\upsilon$-convergent by assumption. 
We deduce from \eqref{eq:res_cont} that for all $t>0$, $p\in\Gamma_{\upsilon}$ and $x\in X$ it holds that
\[
 p(T(t)x)
=\lim_{n\to\infty}\left(\frac{n}{t}\right)^{n}p\left(R\left(\frac{n}{t},A\right)^{n}x\right)
\underset{\eqref{eq:res_cont}}{\leq} p(x)
\] 
where we used that $p$ is $\upsilon$-continuous for the first equality. Further, for $t=0$ we have $p(T(t)x)=p(x)$. 
We conclude that statement (a) holds. 
\end{proof}

The assumptions of \prettyref{prop:gen_bi_cont_contraction} (a) are up to rescaling fulfilled for any 
$\tau$-bi-continuous semigroup on $X$ if it is $\upsilon$-equicontinuous for $\upsilon=\tau$ or $\gamma_{s}$ or $\gamma$ 
(the assumption on $\upsilon$-equicontinuity may fail if $\upsilon=\tau$ by \prettyref{rem:bi_dissip_not_reasonable}).

\begin{rem}
Let $(X,\|\cdot\|,\tau)$ be a sequentially complete Saks space and $\upsilon=\tau$, $\gamma_{s}$ or $\gamma$. 
Then there exists a norming directed system of continuous seminorms that generates the topology $\upsilon$ 
by \prettyref{defn:mixed_top_Saks} (c) for $\upsilon=\tau$, by \prettyref{rem:submixed_norming} for $\upsilon=\gamma_s$ 
and by \prettyref{rem:mixed_norming} for $\upsilon=\gamma$. 
Further, let $(T(t))_{t\geq 0}$ be a $\tau$-bi-continuous, $\upsilon$-equicontinuous semigroup on $X$ 
and $\omega\in\R$ be its type (see \prettyref{defn:bi_continuous} (iii)).
By modifying the proof of \cite[Remark 3.12, p.~9--10]{budde_wegner2022} one can show that for the $\tau$-bi-continuous, 
$\upsilon$-equicontinuous contraction semigroup $(e^{-\omega t}T(t))_{t\geq 0}$ on $X$ there exists a norming
directed system of continuous seminorms $\Gamma_{\upsilon}$ that generates $\upsilon$ such that 
$p(e^{-\omega t}T(t)x)\leq p(x)$ for all $t\geq 0$, $p\in\Gamma_{\upsilon}$ and $x\in X$. 
\end{rem}

In addition, we have the following characterisation in the case of complete, C-sequential Saks spaces $(X,\|\cdot\|,\tau)$ 
and $\upsilon=\gamma$. 

\begin{prop}\label{prop:mixed_equi_diss}
Let $(X,\|\cdot\|,\tau)$ be a complete, C-sequential Saks space and $(A,D(A))$ the generator of a $\tau$-bi-continuous semigroup 
$(T(t))_{t\geq 0}$ on $X$. Then the following assertions are eqivalent:
\begin{enumerate}
\item $(T(t))_{t\geq 0}$ is $\gamma$-equicontinuous. 
\item There is a directed system of continuous seminorms $\Gamma_{\gamma}$ that generates the mixed topology $\gamma$ such that
$(A,D(A))$ is $\Gamma_{\gamma}$-dissipative. 
\end{enumerate}
\end{prop}
\begin{proof}
By \prettyref{rem:relation_gamma_bi_cont_sg} we know that $(T(t))_{t\geq 0}$ is quasi-$\gamma$-equicontinuous 
if $(X,\|\cdot\|,\tau)$ is a sequentially complete C-sequential Saks space. Hence the equivalence of the assertions (a) and (b) 
follows from \cite[Propositions 4.2, 4.4, p.~933, 935]{albanese2016}. 
\end{proof}

The condition that $(X,\|\cdot\|,\tau)$ is a complete, C-sequential Saks space is quite often fulfilled, 
e.g.~for the examples from \prettyref{rem:complete_saks} under some minor constraints by 
\cite[Remark 3.19, p.~14]{kruse_schwenninger2022}, \cite[Remark 3.20 (c), p.~15]{kruse_schwenninger2022}, 
\cite[Example 4.12, p.~24--25]{kruse_schwenninger2022} and \cite[Corollary 3.23 (b), p.~17]{kruse_schwenninger2022}.

\begin{rem}\label{rem:C_sequential_Saks}
\begin{enumerate}
\item Let $\Omega$ be a hemicompact Hausdorff $k_{\R}$-space or a Polish space. 
Then $(\mathrm{C}_{\operatorname{b}}(\Omega),\|\cdot\|_{\infty},\tau_{\operatorname{co}})$ is a complete, C-sequential Saks space.
\item Let $(X,\|\cdot\|)$ be a separable Banach space.
Then $(X',\|\cdot\|_{X'},\sigma^{\ast})$ is a complete, C-sequential Saks space.
\item Let $(X,\|\cdot\|)$ be an SWCG space (see \cite[p.~387]{schluechtermann1988}), 
or a sequentially $\sigma(X,X')$-complete space with an almost shrinking basis (see \cite[p.~75]{kalton1973}). 
Then the triple $(X',\|\cdot\|_{X'},\mu^{\ast})$ is a complete, C-sequential Saks space. 
\item Let $(X,\|\cdot\|_{X})$ be a separable Banach space and $(Y,\|\cdot\|_{Y})$ a Banach space.
Then $(\mathcal{L}(X;Y),\|\cdot\|_{\mathcal{L}(X;Y)},\tau_{\operatorname{sot}})$ is a complete, C-sequential Saks space.
\item Let $H$ be a separable Hilbert space. 
Then $(\mathcal{L}(H),\|\cdot\|_{\mathcal{L}(H)},\tau_{\operatorname{sot}^{\ast}})$ 
is a complete, C-sequential Saks space.
\item Let $\Omega$ be a Polish space. 
Then $(\mathrm{M}_{\operatorname{t}}(\Omega),\|\cdot\|_{\mathrm{M}_{\operatorname{t}}(\Omega)},\sigma(\mathrm{M}_{\operatorname{t}}(\Omega),\mathrm{C}_{\operatorname{b}}(\Omega)))$ is a complete, C-sequential Saks space.
\end{enumerate}
\end{rem}

\begin{rem}\label{rem:bi_dissip_not_reasonable}
Let $(X,\|\cdot\|,\tau)$ be a complete, C-sequential Saks space.
Due to \prettyref{rem:relation_gamma_bi_cont_sg} assertion (a) of \prettyref{prop:mixed_equi_diss} always holds up to 
rescaling, and thus for any $\tau$-bi-continuous semigroup $(T(t))_{t\geq 0}$ on $X$ there is a rescaling such that 
the generator of the rescaled semigroup is $\Gamma_{\gamma}$-dissipative for some system of continuous seminorms 
$\Gamma_{\gamma}$ that generates the mixed topology $\gamma$. 

On the other hand, there are important examples of $\tau$-bi-continuous semigroups that are not quasi-$\tau$-equicontinuous. 
For instance, the Gau{\ss}--Weierstra{\ss} semigroup on the complete, C-sequential Saks space 
$(\mathrm{C}_{\operatorname{b}}(\R^{d}),\|\cdot\|_{\infty},\tau_{\operatorname{co}})$ is $\tau_{\operatorname{co}}$-bi-continuous but 
not locally $\tau_{\operatorname{co}}$-equicontinuous by \cite[Examples 6 (a), p.~209--210]{kuehnemund2003}. 
Since there is some $\lambda>0$ such that $\lambda-A$ is surjective for the generator $(A,D(A))$ 
of the Gau{\ss}--Weierstra{\ss} semigroup by \cite[Lemma 7, Proposition 8, Theorem 12, p.~211--212, 215]{kuehnemund2003}, 
it follows from \prettyref{prop:gen_bi_cont_contraction} with 
$\upsilon=\tau_{\operatorname{co}}$ that its generator $(A,D(A))$ is not bi-dissipative 
(even after rescaling, cf.~\cite[Example 3.9, p.~8]{budde_wegner2022}). 
Another example of a $\tau_{\operatorname{co}}$-bi-continuous semigroup which has no bi-dissipative generator 
(even after rescaling) is the left translation semigroup on the complete, C-sequential Saks space 
$(\mathrm{C}_{\operatorname{b}}(\R),\|\cdot\|_{\infty},\tau_{\operatorname{co}})$ which 
is $\tau_{\operatorname{co}}$-bi-continuous, even locally $\tau_{\operatorname{co}}$-equicontinuous but not 
quasi-$\tau_{\operatorname{co}}$-equicontinuous by \cite[Examples 6 (b), p.~209--210]{kuehnemund2003} and 
\cite[Example 3.2, p.~549]{kunze2009}. 

Nevertheless, the Gau{\ss}--Weierstra{\ss} semigroup and the left translation semigroup are 
quasi-$\beta_{0}$-equicontinuous by \prettyref{rem:relation_gamma_bi_cont_sg} and \prettyref{rem:C_sequential_Saks} (a), 
and thus both have (after rescaling) a $\Gamma_{\beta_{0}}$-dissipative generator for some system of continuous seminorms 
$\Gamma_{\beta_{0}}$ that generates the mixed topology $\beta_{0}=\gamma(\|\cdot\|_{\infty},\tau_{\operatorname{co}})$. 
This underlines that in the framework of $\tau$-bi-continuous semigroups the concept of a bi-dissipative operator resp.~generator 
is not the correct choice whereas the concept of a $\Gamma_{\gamma}$-dissipative operator resp.~generator is the 
more reasonable one. 
\end{rem}

\subsection*{Acknowledgement}
We would like to thank David Jornet for his kind explanations regarding the proof of 
\cite[Proposition 3.11 (iii), p.~927]{albanese2016}.

\bibliography{biblio_lumer_phillips}

\begin{thebibliography}{34}
\providecommand{\natexlab}[1]{#1}
\providecommand{\url}[1]{\texttt{#1}}
\expandafter\ifx\csname urlstyle\endcsname\relax
  \providecommand{\doi}[1]{doi: #1}\else
  \providecommand{\doi}{doi: \begingroup \urlstyle{rm}\Url}\fi

\bibitem[Albanese and Jornet(2016)]{albanese2016}
A.A. Albanese and D.~Jornet.
\newblock Dissipative operators and additive perturbations in locally convex
  spaces.
\newblock \emph{Math. Nachr.}, 289\penalty0 (8-9):\penalty0 920--949, 2016.
\newblock \doi{10.1002/mana.201500150}.

\bibitem[Budde(2021{\natexlab{a}})]{budde2021}
C.~Budde.
\newblock {Positive Miyadera--Voigt perturbations of bi-continuous semigroups}.
\newblock \emph{Positivity}, 25\penalty0 (3):\penalty0 1107--1129,
  2021{\natexlab{a}}.
\newblock \doi{10.1007/s11117-020-00806-1}.

\bibitem[Budde(2021{\natexlab{b}})]{budde2021a}
C.~Budde.
\newblock {Positive Desch--Schappacher perturbations of bi-continuous
  semigroups on \textrm{AM}-spaces}.
\newblock \emph{Acta Sci. Math. (Szeged)}, 87\penalty0 (3-4):\penalty0
  571--594, 2021{\natexlab{b}}.
\newblock \doi{10.14232/actasm-021-914-5}.

\bibitem[Budde and Farkas(2019)]{budde2019}
C.~Budde and B.~Farkas.
\newblock Intermediate and extrapolated spaces for bi-continuous operator
  semigroups.
\newblock \emph{J. Evol. Equ.}, 19\penalty0 (2):\penalty0 321--359, 2019.
\newblock \doi{10.1007/s00028-018-0477-8}.

\bibitem[Budde and Wegner(2022 (Online first))]{budde_wegner2022}
C.~Budde and S.-A. Wegner.
\newblock A {L}umer--{P}hillips type generation theorem for bi-continuous
  semigroups.
\newblock \emph{Z. Anal. Anwend.}, pages 1--16, 2022 (Online first).
\newblock \doi{10.4171/ZAA/1695}.

\bibitem[Conway(1967)]{conway1967}
J.B. Conway.
\newblock {The strict topology and compactness in the space of measures. II}.
\newblock \emph{Trans. Amer. Math. Soc.}, 126\penalty0 (3):\penalty0 474--486,
  1967.
\newblock \doi{10.1090/S0002-9947-1967-0206685-2}.

\bibitem[Cooper(1978)]{cooper1978}
J.B. Cooper.
\newblock \emph{Saks spaces and applications to functional analysis}.
\newblock North-Holland Math. Stud. 28. North-Holland, Amsterdam, 1978.

\bibitem[Engel and Nagel(2000)]{engel_nagel2000}
K.-J. Engel and R.~Nagel.
\newblock \emph{One-parameter semigroups for linear evolution equations}.
\newblock Grad. Texts in Math. 194. Springer, New York, 2000.
\newblock \doi{10.1007/b97696}.

\bibitem[Es-Sarhir and Farkas(2006)]{es_sarhir2006}
A.~Es-Sarhir and B.~Farkas.
\newblock Perturbation for a class of transition semigroups on the {H}\"older
  space {$C_{b,\operatorname{loc}}^{\theta}(H)$}.
\newblock \emph{J. Math. Anal. Appl.}, 315\penalty0 (2):\penalty0 666--685,
  2006.
\newblock \doi{10.1016/j.jmaa.2005.04.024}.

\bibitem[Fabian et~al.(2011)Fabian, Habala, H{\'a}jek, Montesinos, and
  Zizler]{fabian2011}
M.~Fabian, P.~Habala, P.~H{\'a}jek, V.~Montesinos, and V.~Zizler.
\newblock \emph{Banach space theory: The basis for linear and nonlinear
  analysis}.
\newblock CMS Books Math. Springer, New York, 2011.
\newblock \doi{10.1007/978-1-4419-7515-7}.

\bibitem[Farkas(2003)]{farkas2003}
B.~Farkas.
\newblock \emph{Perturbations of bi-continuous semigroups}.
\newblock PhD thesis, E\"otv\"os Lor\'and University, Budapest, 2003.

\bibitem[Farkas(2004)]{farkas2004}
B.~Farkas.
\newblock Perturbations of bi-continuous semigroups with applications to
  transition semigroups on {$C_{b}(H)$}.
\newblock \emph{Semigroup Forum}, 68\penalty0 (1):\penalty0 87--107, 2004.
\newblock \doi{10.1007/s00233-002-0024-2}.

\bibitem[Farkas(2011)]{farkas2011}
B.~Farkas.
\newblock Adjoint bi-continuous semigroups and semigroups on the space of
  measures.
\newblock \emph{Czech. Math. J.}, 61\penalty0 (2):\penalty0 309--322, 2011.
\newblock \doi{10.1007/s10587-011-0076-0}.

\bibitem[Feller(1952)]{feller1952}
W.~Feller.
\newblock The parabolic differential equations and the associated semi-groups
  of transformations.
\newblock \emph{Ann. of Math. (2)}, 55\penalty0 (3):\penalty0 468--519, 1952.
\newblock \doi{10.2307/1969644}.

\bibitem[Gl\"uck et~al.(2022)Gl\"uck, Nagel, and
  Remizov~(eds.)]{open_problems_OPSO2022}
J.~Gl\"uck, R.~Nagel, and I.~Remizov~(eds.).
\newblock {Open problems in operator semigroup theory (OPSO 2021 and 2022),
  Part 2}, 2022.
\newblock
  \url{https://nnov.hse.ru/data/2022/02/18/1747622322/Open_problems_in_OPSO_theory_part_2.pdf#page=4&zoom=auto,-466,795}
  (June 01, 2022).

\bibitem[Hille(1948)]{hille1948}
E.~Hille.
\newblock \emph{Functional analysis and semi-groups}.
\newblock Colloquium Publications 31. AMS, Providence, RI, 1948.

\bibitem[Kalton(1973)]{kalton1973}
N.J. Kalton.
\newblock Mackey duals and almost shrinking bases.
\newblock \emph{Math. Proc. Cambridge Philos. Soc.}, 74\penalty0 (1):\penalty0
  73--81, 1973.
\newblock \doi{10.1017/S0305004100047812}.

\bibitem[Kruse and Schwenninger(2022{\natexlab{a}})]{kruse_schwenninger2022}
K.~Kruse and F.L. Schwenninger.
\newblock On equicontinuity and tightness of bi-continuous semigroups.
\newblock \emph{J. Math. Anal. Appl.}, 509\penalty0 (2):\penalty0 1--27,
  2022{\natexlab{a}}.
\newblock \doi{10.1016/j.jmaa.2021.125985}.

\bibitem[Kruse and Schwenninger(2022{\natexlab{b}})]{kruse_schwenninger2022a}
K.~Kruse and F.L. Schwenninger.
\newblock Sun dual theory for bi-continuous semigroups, 2022{\natexlab{b}}.
\newblock arXiv preprint \url{https://arxiv.org/abs/2203.12765v2}.

\bibitem[Kruse and Seifert(2022)]{kruse_seifert2022a}
K.~Kruse and C.~Seifert.
\newblock Final state observability and cost-uniform approximate
  null-controllability for bi-continuous semigroups, 2022.
\newblock arXiv preprint \url{https://arxiv.org/abs/2206.00562v1}.

\bibitem[Kruse et~al.(2021)Kruse, Meichsner, and
  Seifert]{kruse_meichnser_seifert2018}
K.~Kruse, J.~Meichsner, and C.~Seifert.
\newblock Subordination for sequentially equicontinuous equibounded
  {$C_0$}-semigroups.
\newblock \emph{J. Evol. Equ.}, 21\penalty0 (2):\penalty0 2665--2690, 2021.
\newblock \doi{10.1007/s00028-021-00700-7}.

\bibitem[K\"uhnemund(2001)]{kuehnemund2001}
F.~K\"uhnemund.
\newblock \emph{Bi-continuous semigroups on spaces with two topologies: Theory
  and applications}.
\newblock PhD thesis, Eberhard-Karls-Universit\"at T\"ubingen, 2001.

\bibitem[K\"uhnemund(2003)]{kuehnemund2003}
F.~K\"uhnemund.
\newblock {A Hille--Yosida theorem for bi-continuous semigroups}.
\newblock \emph{Semigroup Forum}, 67\penalty0 (2):\penalty0 205--225, 2003.
\newblock \doi{10.1007/s00233-002-5000-3}.

\bibitem[Kunze(2009)]{kunze2009}
M.~Kunze.
\newblock Continuity and equicontinuity of semigroups on norming dual pairs.
\newblock \emph{Semigroup Forum}, 79\penalty0 (3):\penalty0 540--560, 2009.
\newblock \doi{10.1007/s00233-009-9174-9}.

\bibitem[Lumer and Phillips(1961)]{lumer_phillips1961}
G.~Lumer and R.S. Phillips.
\newblock Dissipative operators in a {B}anach space.
\newblock \emph{Pacific J. Math.}, 11\penalty0 (2):\penalty0 679--698, 1961.
\newblock \doi{10.2140/pjm.1961.11.679}.

\bibitem[Meise and Vogt(1997)]{meisevogt1997}
R.~Meise and D.~Vogt.
\newblock \emph{Introduction to functional analysis}.
\newblock Oxf. Grad. Texts Math. 2. Clarendon Press, Oxford, 1997.

\bibitem[Michael(1973)]{michael1973}
E.A. Michael.
\newblock On $k$-spaces, $k_{R}$-spaces and {$k(X)$}.
\newblock \emph{Pacific J. Math.}, 47\penalty0 (2):\penalty0 487--498, 1973.
\newblock \doi{10.2140/pjm.1973.47.487}.

\bibitem[Miyadera(1952)]{miyadera1952}
I.~Miyadera.
\newblock Generation of a strongly continuous semi-group operators.
\newblock \emph{Tohoku Math. J. (2)}, 4\penalty0 (2):\penalty0 109--114, 1952.
\newblock \doi{10.2748/tmj/1178245412}.

\bibitem[Phillips(1952)]{phillips1952}
R.S. Phillips.
\newblock On the generation of semigroups of linear operators.
\newblock \emph{Pacific J. Math.}, 2\penalty0 (3):\penalty0 343--369, 1952.
\newblock \doi{10.2140/pjm.1952.2.343}.

\bibitem[Schl\"uchtermann and Wheeler(1988)]{schluechtermann1988}
G.~Schl\"uchtermann and R.F. Wheeler.
\newblock On strongly {WCG Banach} spaces.
\newblock \emph{Math. Z.}, 199\penalty0 (3):\penalty0 387--398, 1988.
\newblock \doi{10.1007/BF01159786}.

\bibitem[Snipes(1973)]{snipes1973}
R.F. Snipes.
\newblock C-sequential and {S}-bornological topological vector spaces.
\newblock \emph{Math. Ann.}, 202\penalty0 (4):\penalty0 273--283, 1973.
\newblock \doi{10.1007/BF01433457}.

\bibitem[Wilansky(1981)]{wilansky1981}
A.~Wilansky.
\newblock Mazur spaces.
\newblock \emph{Int. J. Math. Math. Sci.}, 4\penalty0 (1):\penalty0 39--53,
  1981.
\newblock \doi{10.1155/S0161171281000021}.

\bibitem[Wiweger(1961)]{wiweger1961}
A.~Wiweger.
\newblock Linear spaces with mixed topology.
\newblock \emph{Studia Math.}, 20\penalty0 (1):\penalty0 47--68, 1961.
\newblock \doi{10.4064/sm-20-1-47-68}.

\bibitem[Yosida(1948)]{yosida1948}
K.~Yosida.
\newblock On the differentiability and the representation of one-parameter
  semi-group of linear operators.
\newblock \emph{J. Math. Soc. Japan}, 1\penalty0 (1):\penalty0 15--21, 1948.
\newblock \doi{10.2969/jmsj/00110015}.

\end{thebibliography}
\bibliographystyle{plainnat}
\end{document}